\documentclass[preprint]{imsart}

\RequirePackage[OT1]{fontenc}
\RequirePackage{amsthm, amsmath, amssymb, mathrsfs}
\RequirePackage[numbers]{natbib}
\usepackage{bm, bbm}

\usepackage[toc,page]{appendix}

\makeatletter
\newcommand*{\addFileDependency}[1]{
  \typeout{(#1)}
  \@addtofilelist{#1}
  \IfFileExists{#1}{}{\typeout{No file #1.}}
}
\makeatother

\startlocaldefs
\numberwithin{equation}{section}
\theoremstyle{plain}
\newtheorem{thm}{Theorem}[section]
\newtheorem{lem}{Lemma}[section]
\newtheorem{corollary}{Corollary}[section]
\newtheorem{remark}{Remark}[section]
\newtheorem{definition}{Definition}[section]
\newtheorem{Condition}{Condition}[section]
\newtheorem{proposition}{Proposition}[section]

\newcommand{\EE}[1]{\mathbb E\left[#1\right]}
\newcommand{\PP}[1]{\mathbb P\left[#1\right]}

\newcommand{\ind}{\mathbbm 1}

\endlocaldefs

\begin{document}

\begin{frontmatter}
\title{Testing partial conjunction hypotheses under dependency, with applications to meta-analysis} 
\runtitle{Testing partial conjunction hypotheses under dependency}

\begin{aug}
\author{\fnms{Marina} \snm{Bogomolov}\ead[label=e1]{marinabo@technion.ac.il}}

\address{Faculty of Industrial Engineering and Management, Technion - Israel Institute of Technology,
Haifa, Israel\\
\printead{e1}}



\runauthor{M. Bogomolov}


\end{aug}




\begin{abstract}
    In many statistical problems the hypotheses are naturally divided into groups, and the investigators are interested to perform group-level inference, possibly along with inference on individual hypotheses.  We consider the goal of discovering groups containing $u$ or more signals  with group-level false discovery rate (FDR) control. This goal can be addressed by multiple testing of partial conjunction hypotheses with a parameter $u,$ which reduce to global null hypotheses for $u=1.$ We consider the case where the partial conjunction $p$-values are combinations of within-group $p$-values, and obtain sufficient conditions on (1) the dependencies among the $p$-values within and across the groups, (2) the combining method for obtaining partial conjunction $p$-values, and (3) the multiple testing procedure, for obtaining FDR control on partial conjunction discoveries. We consider separately the dependencies encountered in the meta-analysis setting, where multiple features are tested in several independent studies, and the $p$-values within each study may be dependent. Based on the results for this setting, we generalize the procedure of Benjamini, Heller, and Yekutieli (2009) for assessing replicability of signals across studies, and extend their theoretical results regarding FDR control with respect to replicability claims.
\end{abstract}
\end{frontmatter}
\section{Introduction}\label{sec-intro}
In modern research, the investigators often face complex statistical problems, which involve testing a large number of hypotheses. In many applications the hypotheses have a natural group structure, and identifying groups containing signals, (i.e. false null hypotheses), is often of interest. 
For example, in functional magnetic resonance imaging, several types of division into groups have been considered. 
If the goal is to find locations (voxels) of activation while a subject performs a certain cognitive task, each hypothesis corresponds to a single location. The division of the brain into anatomical or functional regions may define groups of hypotheses, and the researchers may be interested to find brain regions containing activated locations (see, for example, \cite{PennyFriston} and \cite{BH07}). 
If several cognitive tasks are studied, one may be interested to find brain locations that are activated in at least two (or, more generally, at least a pre-specified number $u$ of) cognitive tasks. In this case the hypotheses regarding the activation in the same location in different tasks define a group, and the goal is to find groups (i.e. locations) with at least $u$ signals. Motivated by this application, Benjamini and Heller \cite{BH08} introduced the \textit{partial conjunction hypothesis} testing problem, which is defined for group $g$ as
$$H_0^{u/n}(g): k_g<u \text{     versus     } H_1^{u/n}(g): k_g\geq u,$$
where $k_g$ is the number of false null hypotheses in group $g.$ 
In the case where $u=1,$ the partial conjunction null reduces to the global null, or the intersection hypothesis, which holds if all the null hypotheses in the group are true. Benjamini and Heller \cite{BH08} developed a general method for combining  the  $p$-values for the elementary hypotheses in group $g,$  so that the combined $p$-value is valid for testing $H_0^{u/n}(g).$ 
 For the fMRI application above, \cite{BH08} suggested identifying the brain locations which are activated in at least $u$ out of $n$ cognitive tasks by multiple testing of the 
 family $\{H_0^{u/n}(g), \,\,g=1, \ldots G\}$, where $G$ is the number of brain locations considered. In this setting, the false discovery corresponds to rejecting $H_0^{u/n}(g)$ for some group $g$ which contains less than $u$ signals, which, in this application, translates to erroneously stating that the brain location $g,$ with $k_g<u,$ is activated in at least $u$ cognitive tasks.  The false discovery rate (FDR, \cite{BH95}) and the familywise error rate (FWER) for the family of partial conjunction hypotheses are defined with respect to these erroneous group discoveries. Other examples of inference at the group level in fMRI research were given in \cite{FBR15} and \cite{BB14}, who developed different algorithms incorporating group-level inference. Besides fMRI research, there are many other areas where multiple testing of partial conjunction hypotheses has been considered, for example in genomic research, where several hypotheses are tested for each genetic variant (see, e.g., \cite{PetS16}, \cite{HetS16}). 
 
 An important statistical challenge where multiple testing of partial conjunction hypotheses may be very useful is that of meta-analysis. Consider the setting where the same set of features is tested in several independent studies. The hypotheses tested in different studies regarding the same feature may be considered as a group. The global null hypothesis corresponding to a single feature addresses the question: "Is this feature non-null (i.e. corresponds to a false null hypothesis) in at least one of the studies?".
 Identifying the features with a positive answer to this question corresponds to the typical goal of meta-analysis. However, sometimes the evidence that the feature is non-null in at least one study is not enough, and the researchers wish to identify features with consistent findings across studies, which corresponds to the goal of replicability analysis (\cite{BHY09}, \cite{heller2014deciding}). In this case the interest may be to find features which are non-null in at least $u\geq 2$ studies. Benjamini et al. \cite{BHY09} suggested addressing this goal by testing the family $\{H_0^{u/n}(g), \,\,g=1, \ldots, G\},$  where $n$ is the number of studies and $G$ is the number of features considered. The larger is the value of $u,$ the stronger is our requirement for claiming replicability. Since it may be not clear what value of $u$ should be chosen, \cite{BHY09} suggested to select the features with potential replicability findings by applying the Benjamini-Hochberg procedure (BH, \cite{BH95}) on the family of global null $p$-values for all the features, and assess the replicability strength for each selected feature $g$ by a lower bound on the true number of studies where the feature is non-null, $\hat{k}(g),$ obtained from testing the partial conjunction hypotheses in order with $u=1, \ldots, n.$ 
 The authors proved that their procedure controls the expected proportion of features with erroneously reported lower bounds (i.e. features with $\hat{k}(g)>k(g)$), out of all features selected, under independence within each study. 

In this paper, we consider multiple testing of partial conjunction hypotheses with FDR or weighted FDR (\cite{BH97}) control.  
The main goal of this paper is to develop sufficient conditions on (1) the multiple testing procedure, (2) the dependency structure among the $p$-values for the elementary hypotheses (referred to as elementary $p$-values hereafter),  and (3) the method for obtaining partial conjunction $p$-values,  for (weighted) FDR  control on the family of partial conjunction hypotheses. 
 This goal is addressed by building upon the results of \cite{blanchard2008two}, 
 \cite{blan2009}, \cite{RetJ17}, \cite{BH08}, and \cite{BHY09}. 
 We consider separately testing the partial conjunction hypotheses with (weighted) FDR control under the dependencies of the meta-analysis setting. 
 We generalize the theoretical results of \cite{BHY09} for their procedure for assessing the strength of replicability for the selected features, showing that this procedure provides the above error rate control regarding false replicability claims not only under independence within each study, but also under a certain form of  positive  dependence among the $p$-values within each study, which is typically considered in multiple testing literature and is defined below.
 \begin{definition}[PRDS, \cite{BY01}]
 The vector $X=(X_1, \ldots, X_m)$ is positive regression dependent on a subset $I_0\subseteq\{1, \ldots, m\}$ if for any non-decreasing set $D,$ and for each $i\in I_0,$ $\PP{X\in D\,|\,X_i=x}$ is non-decreasing in $x.$
 \end{definition}
 In addition, we extend the procedure of \cite{BHY09} for addressing (1) a wide class of selection rules for selecting the promising features, (2) weights reflecting prior knowledge on the replicability extent of the features, as well as weights reflecting different penalties for erroneous replicability claims (similar to \cite{genovese2006false} and \cite{BH97}, respectively), and obtain the same error control guarantee for the extended procedure. For arbitrary dependence within each study, a more conservative variant of this procedure is suggested. 
 The remaining contributions of this paper are described in more detail below.


For common multiple testing procedures, the question ``What are the dependencies among the $p$-values, for which the procedure guarantees the control of its target error rate?" was typically addressed. For example, according to the results of \cite{BY01}, the BH procedure  guarantees FDR control when applied on $p$-values which satisfy the PRDS property (its FDR control under independence was initially proved by \cite{BH95}). 
However, when this type of dependence holds among the elementary $p$-values, it is not clear whether it holds among the partial conjunction $p$-values, which may be complicated functions of the within-group elementary $p$-values.
As a special case, consider global null $p$-values based on the method of Simes \cite{Simes1986} (see Section \ref{sec-GN} for the definition). Based on the results of \cite{BY01}, the BH procedure applied on Simes' $p$-values controls the FDR on the family of global null hypotheses if Simes' $p$-values satisfy the PRDS property, however it is not clear for what dependency structures among the  elementary $p$-values this condition holds. 
In fact, in the lecture devoted to Prof. Yoav Benjamini's 70th birthday, which was held in Israel in 2018, Prof. Abba Krieger showed that when the elementary $p$-values satisfy the PRDS property, the global null Simes' $p$-values may not satisfy this property. However, as Prof. Krieger emphasized, this does not yield that the BH procedure applied on Simes' global null $p$-values does not control the FDR in this case, because the PRDS property is a sufficient, not a necessary condition, for FDR control of the BH procedure.  Interestingly, indeed, the results of \cite{RetJ17} yield that the BH procedure applied on Simes' $p$-values guarantees FDR control on the family of global null hypotheses if the set of elementary $p$-values satisfies the PRDS property. Benjamini and Heller \cite{BH08} showed a similar result for the BH procedure applied on partial conjunction $p$-values based on Fisher's or Stouffer's methods, assuming a certain dependency structure among the elementary  $p$-values. 
In this work these results are generalized with respect to several aspects, which are described below.

 Sufficient conditions on the dependency among the elementary $p$-values for FDR control on partial conjunction hypotheses are developed for a general class of \textit{self-consistent} procedures, defined by \cite{blanchard2008two}, which includes the BH procedure, as well as its weighted variants (see \cite{BH97}, \cite{genovese2006false}, and \cite{blanchard2008two}). This class also includes all procedures which always reject a subset of the hypotheses rejected by the BH procedure. Although these procedures are more conservative than the BH procedure, they may be of interest due to structural constraints on the rejection set, which may be induced by spatial or graph structure of the hypotheses, see, e.g., \cite{filtering} and \cite{BetS17a}. 

Besides the PRDS dependence among the $p$-values for the initial set of elementary hypotheses, more complex forms of dependencies are considered. Specifically, we consider a certain form of positive dependence across groups and independence within each group, as well as positive dependence across groups and arbitrary dependence within the groups. Independence and arbitrary dependence across the groups are also considered. Finally, 
we consider the dependencies which may hold in meta-analysis, where each feature is tested in several independent studies, addressing independence, positive dependence, and arbitrary dependence within each study. 
The results of this paper show that for most of the above types of  dependencies across and within groups, the BH procedure, as well as its weighted variants,  guarantees FDR control when applied on certain valid partial conjunction $p$-values. Therefore, despite the complex dependencies among the partial conjunction $p$-values in these situations, which may not satisfy the PRDS property,  there is no need for conservative adjustments that were suggested for arbitrary dependence, such as the adjustment of \cite{BY01}, or its generalization given by \cite{blanchard2008two}.

The sufficient conditions that are obtained for FDR control depend on the method for obtaining partial conjunction $p$-values. We consider partial conjunction $p$-values which are connected to self-consistent multiple testing procedures, such as those based on Simes' and Bonferroni methods. In particular, we address the global null $p$-values connected to the adaptive self-consistent procedures, such as the adaptive BH procedure incorporating Storey's estimator for the proportion of nulls \cite{SetS04}, and their extensions to general partial conjunction $p$-values, based on the method of \cite{BH08}. For the meta-analysis setting, we also consider the popular Fisher's (\cite{fisher1932statistical}) and Stouffer's (\cite{stouffer1949american}) methods for combining $p$-values for testing the partial conjunction null. 

For obtaining the results above, several lemmas in the spirit of the group-level superuniformity lemma of \cite{RetJ17} are obtained. 
These lemmas are used for generalizaing the results of \cite{BHY09} regarding their procedure for replicability analysis, as described above. The lemmas may be useful for obtaining FDR control on partial conjunction hypotheses for other  procedures, such as
p-filter \cite{RetJ17}, DAGGER \cite{ramdas2019sequential}, or TreeBH \cite{bogomolov2020hypotheses}. The results obtained in this paper may also be useful for obtaining  FWER control guarantees on partial conjunction hypotheses under certain dependency structures among the elementary $p$-values.
This paper is organized as follows. 
Section \ref{sec-agg} includes several  global null tests which are addressed in this work, as well as their extensions for testing partial conjunction hypotheses. 
Section \ref{mult-PC} considers the setting where the hypotheses are divided into natural groups, and gives several sufficient conditions for FDR control on a family of partial conjunction hypotheses defined for these groups, for several dependency structures. 
In Section \ref{meta-analysis} we consider the meta-analysis setting, and give some results for partial conjunction testing in this setting, as well as a generalization of the procedure and the results of \cite{BHY09} addressing replicability analysis. All the proofs are given in the Appendix. Several remarks regarding the proofs, as well as key lemmas, are given in Section \ref{sec-lemmas}. The paper ends by a discussion in Section \ref{discussion}.

\begin{remark}\label{technical}
Throughout this paper, vectors and matrices are denoted by bold letters, and $p$-values for composite hypotheses comprising combinations of elementary $p$-values  are denoted by capital letters. In addition, we define $0/0=0,$ and for two random variables $X$ and $Y$ such that if $Y=0,$ then $X=0$ almost surely, we define $X/Y=0.$
\end{remark}

\section{Aggregate-level testing of a single hypothesis}\label{sec-agg}
\subsection {Testing a global null hypothesis}\label{sec-GN}
Consider a family of $m$ elementary null hypotheses,  $H_1, \ldots, H_m,$ with $p$-values $p_1, \ldots, p_m.$  The global null hypothesis is defined as the intersection of these null hypotheses, stating that all these null hypotheses are true, $H_0=\cap_{i=1}^m H_i.$ 
We consider the cases where the $p$-value for the global null is obtained by combining the $p$-values for the individual hypotheses, $H_1, \ldots, H_m,$ via a certain combining function $f: [0,1]^m\rightarrow [0,1].$
The combination of $p$-values $f(p_1, \ldots, p_m)$ is a valid $p$-value for the global null 
if for any $x\in(0,1),$  $\PP{f(p_1, \ldots, p_m)\leq x}\leq x$ under the global null, i.e. when all the null hypotheses are true.
Several methods for combining the $p$-values $p_1, \ldots, p_m$ for global null testing have been developed and studied, see, e.g., \cite{loughin2004systematic}, \cite{owen2009karl}, \cite{vovk2020combining},  \cite{DJ04}, \cite{wilson2019harmonic}, and references therein.
In this work the focus is on the following methods.


\begin{itemize}\item Fisher's method \cite{fisher1932statistical}:
\begin{align}
f_{\text{Fisher}}(p_1, \ldots, p_m)=\PP{\chi^2_{(2m)}\geq -2\sum_{i=1}^m\log(p_i)},\label{Fisher}
\end{align}
where $\chi^2_{(2m)}$ is a chi-square variable with $2m$ degrees of freedom. 
\item Stouffer's method \cite{stouffer1949american}:
\begin{align}
f_{\text{Stouffer}}(p_1, \ldots, p_m)=1-\Phi\left(\frac{\sum_{i=1}^m\Phi^{-1}(1-p_i)}{\sqrt{m}}\right)\label{Stouffer},
\end{align}
with $\Phi(\cdot)$ denoting the cumulative distribution function of a standard normal variable, and $\Phi^{-1}(\cdot)$ denoting its inverse function.
\item Simes' method \cite{Simes1986}:
\begin{align}f_{\text{Simes}}(p_1, \ldots, p_m)=\min_{k=1, \ldots, m}\frac{mp_{(k)}}{k},\label{Simes}
\end{align}
where $p_{(1)}\leq\ldots\leq p_{(m)}$ is the ordered sequence of $p$-values.
	\end{itemize}
All the methods above give  valid global null $p$-values  under independence of individual $p$-values. Simes' method is known to be also valid under certain forms of positive dependence, see \cite{sarkar1998some}, \cite{BY01}, \cite{finner2017simes}. For general dependence, the following more conservative methods are considered:
\begin{itemize}
\item Bonferroni's method:
\begin{align}
f_{\text{Bonferroni}}(p_1, \ldots, p_m)=\min\{mp_{(1)},\,\, 1\}\label{Bonf}
\end{align}
\item Hommel's method \cite{hommel1983tests}:
\begin{align}
f_{\text{Hommel}}(p_1, \ldots, p_m)=\min\left\{\left(1+\frac{1}{2}+\ldots+\frac{1}{m}\right)\min_{k=1, \ldots, m}\frac{mp_{(k)}}{k},\,\, 1\right\}\label{Hommel}
\end{align}
\end{itemize}

Some of the methods for global null testing are connected to certain multiple testing procedures. For example, if one uses Bonferroni's method (\ref{Bonf}) for testing the global null,  then the global null is rejected at level $\alpha$  if and only if Bonferroni's procedure, applied at level $\alpha$ on the $p$-values $p_1, \ldots, p_m,$ rejects at least one hypothesis.  There is a similar connection between Simes' $p$-value (\ref{Simes}) and the BH procedure, as well as
Hommel's $p$-value (\ref{Hommel}) and the procedure of \cite{BY01} (BY), which is  a conservative variant of the BH procedure with FDR control guarantees under arbitrary dependence. In all these cases, the global null $p$-value is the minimal adjusted $p$-value of a certain multiple testing procedure. Recall that for a given multiple testing procedure, an adjusted $p$-value for $H_i,$ denoted by $P_i^{adj},$ is the minimum level $\alpha$ of the procedure such that $H_i$ is rejected at level $\alpha$ (\cite{wright1992adjusted}). The set of rejected hypotheses by a given level-$\alpha$ procedure is $\{i: P_i^{adj}\leq \alpha\},$ therefore at least one rejection is made if and only if the minimum adjusted $p$-value, $P_{(1)}^{adj},$ is bounded above by $\alpha.$ 
Under the global null, each rejected hypothesis corresponds to a false discovery, therefore in this case $\PP{P_{(1)}^{adj}\leq \alpha}=\PP{(V>0)},$ where $V$ is the number of false discoveries of the given level-$\alpha$ multiple testing procedure. In addition, under the global null, $\PP{(V>0)}=\text{FWER}=\text{FDR}$, which gives the following corollary.
\begin{corollary}\label{connection-cor}
	Assume we are given a multiple testing procedure which guarantees control of FWER (or FDR) when all the null hypotheses are true (i.e. in the weak sense) under a certain dependency structure among the $p$-values. Then, under the same dependency structure, the minimum adjusted $p$-value of this procedure is a valid $p$-value for testing the global null.
\end{corollary}
\begin{remark}\label{monotone}
A natural requirement for a function $f(p_1, \ldots, p_m)$ which gives a global null $p$-value is that $f$ is non-decreasing in each coordinate. In fact, Birnbaum \cite{birnbaum1954combining} showed that when the $p$-values are independent, for each combination function $f(p_1, \ldots, p_m)$ satisfying this requirement one can find an alternative hypothesis against which the combination function $f(p_1, \ldots, p_m)$ gives a best test of the global null. 
This requirement is satisfied when $f$ is based on a certain multiple testing procedure in the sense of Corollary \ref{connection-cor}, assuming that the multiple testing procedure satisfies the following natural monotonicity property: given a set of rejected hypotheses $\mathcal{R}_1$, decreasing a certain $p$-value will result in rejecting the set $\mathcal{R}_2,$ where $\mathcal{R}_1\subseteq \mathcal{R}_2.$ 
\end{remark}
Based on Corollary \ref{connection-cor}, one can obtain validity of global null tests based on the results for the multiple testing procedures which induce these tests. The fact that Bonferroni's and Hommel's global null $p$-values are valid under any dependence of elementary $p$-values may be obtained from the facts that Bonferroni and BY  control the FWER in the weak sense under arbitrary dependence.
Similarly, one can obtain the validity of Simes' method and its weighted versions (see \cite{BH97} and \cite{Hochberg1994}) under PRDS dependence based on the result of \cite{blanchard2008two}, showing that  the doubly-weighted BH procedure guarantees FDR control under this type of positive dependence. 
Similarly, the minimum adjusted $p$-value of an adaptive variant of the BH procedure, incorporating an estimator for the proportion of true null hypotheses, denoted by $\pi_0,$ is a valid global null $p$-value if the dependency among the $p$-values is such that the procedure guarantees FDR control. See \cite{benjamini2006adaptive} and \cite{blan2009} for several adaptive procedures and their performance comparisons. Let us consider the minimum adjusted $p$-value of the adaptive BH procedure incorporating a modification of Storey's estimator (\cite{Storey2002})  for the proportion of nulls, suggested by  \cite{SetS04}. The estimator is based on a tuning parameter $\lambda\in(0,1),$ and is given by 
\begin{align}
    \hat{\pi}_0(\lambda)=\frac{W(\lambda)+1}{(1-\lambda)m},\label{Storey}
\end{align}
with $W(\lambda)=\sum_{i=1}^m \ind(p_i>\lambda),$ where $\ind(\cdot)$ is the indicator function.
The minimum adjusted $p$-value of this procedure is
\begin{align}\label{Storey-Simes}
f_{Simes-Storey}(p_1, \ldots, p_m)=\begin{cases} \min\left\{\min\limits_{k:p_{(k)\leq \lambda}}\left\{m\hat{\pi}_0(\lambda)p_{(k)}/k\right\},\,\,1\right\} & p_{(1)}\leq \lambda\\ 1, &  p_{(1)}> \lambda \end{cases}
\end{align}
It was proved by \cite{SetS04} that  the adaptive BH procedure with Storey's estimator for $\pi_0$ controls the FDR under independence, therefore the test based on its minimum adjusted $p$-value, $f_{Simes-Storey}(p_1, \ldots, p_m),$ is valid under independence. 
In fact, we address a more general class of global null tests, induced by a certain class of multiple testing procedures, including certain adaptive FDR-controlling procedures. According to Remark \ref{monotone}, the tests induced by monotone multiple testing procedures satisfy the desirable monotonicity property.

\subsection {Testing a partial conjunction hypothesis}\label{PC}
Consider a family of $m$ elementary null hypotheses $H_1, \ldots, H_m,$ and let $m_1$ be the number of false null hypotheses in this family. Benjamini and Heller \cite{BH08} considered testing the partial conjunction null hypothesis $H_0^{u/m}: m_1<u$ versus the alternative $H_1^{u/m}: m_1\geq u,$ for a fixed $u\in\{1, \ldots, m\},$ addressing the question whether the family contains at least $u$  false null hypotheses. 
For $u=1,$ the partial conjunction null reduces to the global null, and the tests developed in \cite{BH08}  reduce to familiar tests, while for $u=n,$ their tests lead to rejecting the null if the maximum $p$-value is below the given significance level $\alpha.$ In general, the $p$-value for testing $H_0^{u/m}$ denoted by $P^{u/m},$ is valid if under this partial conjunction null, i.e. in the case where $m_1<u,$ $\PP{P^{u/m}\leq x}\leq x$ for all $x\in[0,1].$ Benjamini and Heller \cite{BH08} developed the following method for constructing valid partial conjunction $p$-values $P^{u/m}:$ given a certain combining function $f$ for the global null hypothesis, which is valid for the dependency among the $p$-values in any subgroup $A\subseteq\{1, \ldots, m\}$ of size $m-u+1$ and which is non-decreasing in each of its components, obtain a pooled $p$-value $P^{u/m}$ by applying the function $f$ on the $m-u+1$ largest $p$-values. Since the combining function $f$ is non-decreasing in each of its components, this method is equivalent to computing the global null combined $p$-value for each subgroup $A\subseteq \{1, \ldots, m\}$ of size $m-u+1$ and taking the maximum of all the combined $p$-values. Denoting by $\bm p_A$ the set of $p$-values for the hypotheses in group $A,$ $P^{u/m}$ can be written as follows:  
$$P^{u/m}=\max\{f(\bm p_A): A\subseteq \{1, \ldots, m\}, |A|=m-u+1\}. $$

See \cite{BH08} for partial conjunction $p$-values based on Fisher's, Stouffer's, Simes', and Bonferroni's methods.
 Let us obtain the  partial conjunction $p$-value for testing $H_0^{u/m},$  based on the combining function (\ref{Storey-Simes}), which is  non-decreasing in each coordinate. For obtaining $P^{u/m},$ we apply the combining function (\ref{Storey-Simes}) on the largest $m-u+1$ $p$-values. Let \begin{align*}
\hat{\pi}_0^{u/m}(\lambda)=\frac{1+\sum_{i=u}^m\ind(p_{(i)}>\lambda)}{(m-u+1)(1-\lambda)},
\end{align*} 
and define $S(\lambda)=\{k\in\{1, \ldots, m-u+1\}: p_{(u-1+k)}\leq \lambda\}.$ Then 
\begin{align}
P^{u/m}_{Simes-Storey}= \begin{cases}\min\left\{\min\limits_{k\in S(\lambda)}\left\{(m-u+1)\hat{\pi}_0^{u/m}(\lambda)p_{(u-1+k)}/k\right\},\,\,1\right\} & p_{(u)}\leq \lambda\\ 1, &  p_{(u)}> \lambda \end{cases}\label{PC-Storey}
\end{align} 
Since the combining function (\ref{Storey-Simes}) gives a valid global null $p$-value under independence among $p_1, \ldots, p_m,$ the partial conjunction $p$-value given above is valid for testing $H_0^{u/m}$ under independence.

\section{Multiple testing of partial conjunction hypotheses}\label{mult-PC}

 \subsection{Preliminaries}\label{setting}

 Assume one has $M$ elementary null hypotheses, $H_1, \ldots, H_M,$ with  $p$-values $p_1, \ldots, p_M, $ respectively. The set of hypotheses is divided into $G$  groups. 
 Let $A_g\subseteq\{1, \ldots, M\}$ be the set of indices of hypotheses belonging to group $g,$ and let $n_g=|A_g|$ be the number of hypotheses in group $g.$ 
 Denote by $k_g$ the (unknown) number of false null hypotheses in group $g.$ 
 
 The researcher is interested in testing the family of $G$ partial conjunction hypotheses for the given groups, $\{H_0^{u_g/n_g},\,\, g=1, \ldots, G\},$ based on partial conjunction $p$-values $\{P_g^{u_g/n_g},\,\, g=1, \ldots, G\},$ where
 for each $g\in\{1, \ldots, G\},$ $P_g^{u_g/n_g}$ 
 is a combination of $p$-values for the hypotheses in group $g.$   When $u_g=1$ for $g=1, \ldots, G,$ this is the family of global null hypotheses. When the group sizes are equal, i.e. $n_g=n$ for $g=1, \ldots, G,$ it seems natural to test partial conjunction hypotheses with a common value $u_g=u$ for each group. However, 
 when the groups are of different sizes and the researcher wishes to identify groups in which at least a proportion $p\in(0,1]$ of the null hypotheses are false, the values of $u_g$ may differ across groups. Let $\mathcal{G}_0\subseteq\{1, \ldots, G\}$ be the set of indices for which the partial conjunction null hypotheses are true, i.e. $\mathcal{G}_0=\{g\in\{1, \ldots, G\}: k_g<u_g\}.$ 
 
 We consider the case where a certain multiple testing procedure receives as input the partial conjunction $p$-values $\{P^{u_g/n_g},\,\, g=1, \ldots, G\},$ and outputs the set indices of rejected partial conjunction hypotheses, $\mathcal{R}=\mathcal{R}(p_1,\ldots, p_M)\subseteq\{1, \ldots, G\}.$ The FWER for the family of partial conjunction hypotheses is 
 $$\PP{|\mathcal{R}\cap\mathcal{G}_0|>0},$$
 while the (weighted) FDR, introduced in \cite{BH97}, for the family of partial conjunction hypotheses is
  \begin{align}\text{FDR}_v^{\text{PC}}:=\EE{\frac{\sum_{g\in \mathcal{G}_0}v_gI(g \in \mathcal{R})}{\sum_{g=1}^Gv_gI(g \in \mathcal{R})}},\label{FDR-PC}\end{align}
 where the weights $\{v_g, \,\,g=1, \ldots, G\}$ reflect the penalties for false rejections of partial conjunction hypotheses, and the rewards for true rejections. 
 The regular  FDR is covered by the case where all the weights are equal to unity. Benjamini and Hochberg \cite{BH97} developed a variant of the BH procedure controlling the weighted FDR. The partial conjunction hypotheses may also be associated with prior weights $(w_1, \ldots, w_G),$ reflecting prior beliefs regarding their truth status, which may be used for increasing the power of a multiple testing procedure, as suggested by (\cite{genovese2006false}). Following  (\cite{blanchard2008two}), we consider the general case where the partial conjunction hypotheses are associated with both types of weights.   
 
 In the following sections, we provide sufficient conditions 
  for (weighted) FDR  control on the family of partial conjunction hypotheses, $\{H_0^{u_g/n_g}, \,\,g=1, \ldots, G\}.$ Let $\bm p^{-g}$ be the vector of $p$-values for the hypotheses not belonging to group $g,$ and for each $i\in A_g,$ let $\bm p_g^{-i}$ be the vector of $p$-values for the hypotheses in group $g,$ excluding $H_i.$
  We consider the following dependency structures among the elementary $p$-values:
  \begin{enumerate}  
  	\item[D1] Independence across groups: the $p$-values in each group are independent of the $p$-values in any other group.
  	\item[D2] Independence within each group: for each $g\in\{1, \ldots, G\},$ the $p$-values in the set $\{p_i,\,\, i\in A_g\}$ are independent.
  	\item[D3] Overall positive dependence: the set of all the $p$-values, $p_1, \ldots, p_M,$ is PRDS on the subset of true null hypotheses.
  	  	\item[D4] Conditional positive dependence across groups: for each $g\in\{1, \ldots, G\},$ and each true null hypothesis $H_i$ in group $g,$ the vector $\bm p$ 
  	  	is PRDS  with respect to $p_i,$ conditionally on $\bm p_g^{-i},$ i.e. $\PP{\bm p\in D\,|\,p_i=x, \bm p_g^{-i}=\bm y}$ is non-decreasing in $x,$ for any non-decreasing set $D\in [0,1]^{M},$ and any vector $\bm y\in [0,1]^{n_g-1}.$
  	  	\item[D5] Positive dependence across groups: for each $g\in\{1, \ldots, G\},$ and each true null hypothesis $H_i$ in group $g,$ the vector $(\bm p^{-g}, p_i),$ consisting of $p_i$ along with  $p$-values for the hypotheses not belonging to group $g$, 
  	  	is PRDS  with respect to $p_i,$ i.e. $\PP{(\bm p^{-g}, p_i)\in D\,|\,p_i=x}$ is non-decreasing in $x,$ for any non-decreasing set $D\in [0,1]^{M-n_g+1}.$ 
  	 \end{enumerate}
   The settings where the dependency across the groups is such that the partial conjunction $p$-values are PRDS on the subset of true partial conjunction nulls, or the dependency among the partial conjunction $p$-values is unspecified, are also considered. The meta-analysis setting, where the $p$-values are independent across studies, and may be dependent within studies, is considered separately in Section \ref{meta-analysis}. 
\subsection{Testing partial conjunction hypotheses with FDR control}
\subsubsection{Preliminary definitions}\label{prel}
Blanchard and Roquain \cite{blanchard2008two} developed two sufficient conditions for FDR control of a multiple testing procedure. In this section we review the definitions of \cite{blanchard2008two} and \cite{blan2009} which shall be used for obtaining sufficient conditions for FDR control on partial conjunction hypotheses. 
 Consider a family of $m$ hypotheses $H_1, \ldots, H_m,$ with  $p$-values $\bm{p}=(p_1, \ldots, p_m),$ which are valid, in the sense that for any true null hypothesis $H_i,$ $\PP{p_i\leq x}\leq x$ for all $x\in [0,1].$ Assume that these hypotheses are associated with a vector of prior weights $\bm  w=(w_1, \ldots, w_m),$ and a vector of penalty weights $\bm v=(v_1, \ldots, v_m).$ Recall that given a rejection set $\mathcal{R}$ and the set of indices of true null hypotheses $M_0\subseteq\{1, \ldots, m\},$ the weighted FDR with penalty weights given by $\bm v$ is defined as 
 $$\text{FDR}_{v}=\EE{\frac{\sum_{i\in M_0}v_i I(i\in \mathcal{R})}{\sum_{i=1}^m v_i I(i\in \mathcal{R})}}.$$ 
 Define $|A|_{\bm v}\equiv \sum_{i=1}^m v_i I(i\in \mathcal{A})$ as the volume of the set $A,$ for any  $A\subseteq\{1, \ldots, m\}.$ 
 Then the weighted FDR can be written as
 $$\EE{\frac{|M_0\cap \mathcal{R}|_{\bm v}}{|\mathcal{R}|_{\bm v}}}.$$
 Note that for 
 $\bm v=(1, \ldots, 1),$ $|\cdot|_{\bm v}$ is simply the cardinality measure, and the weighted FDR reduces to the regular FDR. 
  A threshold collection of a multiple testing procedure used for testing the family $\{H_1, \ldots, H_m\}$  was defined by Blanchard and Roquain \cite{blanchard2008two} as follows. 
\begin{definition}[Blanchard and Roquain, \cite{blanchard2008two}]
	A threshold collection $\Delta$ is a  function
	$$\Delta: (i, r)\in \{1, \ldots, m\}\times \mathbb R^+\mapsto\Delta(i, r)\in \mathbb R^+, $$
	which is non-decreasing in its second variable.
\end{definition}
Following \cite{blanchard2008two}  we consider the factorized  threshold collection of  the following form:
$\forall (i, r)\in\{1, \ldots, m\}\times \mathbb R^+,$
\begin{align}\Delta(i,r)=\alpha w_i\beta(r)/m,\label{non-adapt}\end{align}
where $w_i$ is a prior weight for $H_i,$ $i=1, \ldots, m,$ and $\beta: \mathbb R^+\rightarrow \mathbb R^+$ is a non-decreasing function called a shape function. Following \cite{blan2009}, we also consider an adaptive threshold collection of the following form: $\forall (i, r)\in\{1, \ldots, m\}\times \mathbb R^+,$
\begin{align}
\Delta(i,r)=\Delta(r)=\alpha r \hat{\pi}_0^{-1}/m,\label{adapt1}
\end{align}
and $\hat{\pi}_0$  is the estimator for $\pi_0,$ the proportion of nulls. Given a threshold collection $\Delta,$  define $L_{\Delta}(r)$ for $r\in\mathbb R^+$  as in \cite{blanchard2008two}, i.e. 
$$L_{\Delta}(r)=\{i\in \{1, \ldots, m\}: p_i\leq \Delta(i, r)\}.$$
We are now ready to review several properties of multiple testing procedures, which shall be considered in the following sections. 
\begin{definition}\label{blan} 
Let $\mathcal{R}=\mathcal{R}(\bm{p})\subseteq\{1, \ldots, m\}$  be the set of indices of rejected hypotheses by a given multiple testing procedure. For each $i\in\{1, \ldots, m\},$ let $\bm{p}^{-i}$ be the vector of $p$-values excluding $p_i.$ 
\begin{enumerate} \item A multiple testing procedure is \textbf{self-consistent} with respect to threshold collection $\Delta$ of form (\ref{non-adapt}) or (\ref{adapt1}) if the following inclusion holds almost surely:
\begin{align}\mathcal{R}\subseteq L_{\Delta}(|\mathcal{R}|_v).\label{SC} \end{align}
\item The multiple testing procedure is \textbf{non-increasing} if  $|\mathcal{R}(\bm{p})|_v$ is non-increasing in each $p$-value.
\item The multiple testing procedure is \textbf{stable} if for each $i\in \mathcal{R},$ fixing $\bm{p}^{-i}$ and changing $p_i$ as long as $i\in \mathcal{R}$ will not change the set $\mathcal{R}.$ For such procedure, we denote by $\mathcal{R}^{-i}=\mathcal{R}^{-i}(\bm p^{-i})$  the rejection set which is obtained when the $p$-values for all the hypotheses except $H_i$ are given by $\bm{p}^{-i},$ and $p_i$ has any value such that $H_i$ is rejected.
\item A stable multiple testing procedure is \textbf{concordant} if for each $i\in \{1,\ldots,m\},$ $|\mathcal{R}^{-i}|_v$ is non-increasing in each coordinate of $\bm{p}^{-i}.$ 
\end{enumerate}
\end{definition}
The definitions in items 1 and 2 were given in \cite{blanchard2008two} and \cite{blan2009}. The property in item 3 was defined in \cite{MeRutiJASA}. Properties similar to those  in items 3 and 4 were defined in  \cite{BB14}  and \cite{BY05j}.  
Blanchard and Roquain \cite{blanchard2008two} showed that any step-up and any step-down procedure, or more generally, step-up-down procedure (\cite{tamhane1998generalized}) is self-consistent. The monotonicity defined in item 2 is a very natural property, which seems to be satisfied by all commonly used multiple testing procedures. The BH procedure and its weighted variants, suggested by \cite{BH97} and \cite{genovese2006false}, are special cases of the doubly-weighted BH procedure considered by \cite{blanchard2008two}, which satisfies all the properties above (see Appendix \ref{proofs:corBH} for a proof). 
It is easy to see that any procedure which always rejects a subset of the hypotheses rejected by the BH procedure or one of its weighted variants above is  self-consistent with respect to thresholds of form (\ref{non-adapt}) with identity shape function. Such procedures were developed for addressing structural constraints on the rejection set, which are natural in some applications. Examples are the procedure in \cite{BetS17a} which was suggested for controlling the FDR in genome-wide association studies, and its generalization suggested in \cite{filtering} for FDR control on filtered set of discoveries.
For adaptive multiple testing procedures, let us consider the following conditions on $\hat{\pi}_0,$
the estimator for $\pi_0,$ considered by \cite{blan2009}:
\begin{Condition}\label{cond-adapt}
(a) The estimator $\hat{\pi}_0: [0,1]^m\rightarrow(0, \infty)$ is a measurable, coordinatewise  nondecreasing function. (b) If the $p$-values $(p_1, \ldots, p_m)$ are independent, then for any true null hypothesis $H_i,$ $\EE{1/\hat{\pi}_0(\bm{p}_{0,i})}\leq1/\pi_0,$ where $\bm{p}_{0,i}$ 
is the vector of $p$-values where $p_i$ is replaced by 0.
\end{Condition}
The condition that the estimator $\hat{\pi}_0$ is nondecreasing in each $p$-value
is very natural, and it is satisfied by Storey's estimator (\ref{Storey}). Benjamini, Krieger, and Yekutieli \cite{benjamini2006adaptive} proved that this estimator satisfies part (b) of Condition \ref{cond-adapt} as well.  See Corollary 13 of \cite{blan2009} for other estimators satisfying Condition \ref{cond-adapt}.
In the next sections we rely on the definitions above for formulating sufficient conditions for FDR control on partial conjunction hypotheses. 

 \subsubsection {Results for general valid partial conjunction $p$-values}\label{suffic}
 We assume hereafter that for each $g=1, \ldots, G,$ the partial conjunction $p$-value $P_g^{u_g/n_g}$ is a combination of the $p$-values for the hypotheses belonging to $A_g.$ 
 Then the dependence among the partial conjunction $p$-values is induced by the combining method and the dependence among the elementary $p$-values. The following proposition shows several  sufficient conditions for FDR control on partial conjunction hypotheses for general valid partial conjunction $p$-values. 
  \begin{proposition}\label{main-FDR}
 	Let $\{H_0^{u_g/n_g},\,  g=1, \ldots, G\}$ be a family of partial conjunction hypotheses associated with penalty weights $\{v_g, \,g=1, \ldots, G\}$ and prior weights $\{w_g, \,g=1, \ldots, G\}$ satisfying $\,\sum_{g \in \mathcal{G}}v_gw_g=G.$ Let $\{P_g^{u_g/n_g}, \, g=1, \ldots, G\}$ be the set of corresponding partial conjunction $p$-values, and assume that the dependency among the elementary $p$-values within each group is such that these partial conjunction $p$-values are valid for testing the family $\{H_0^{u_g/n_g}, \,\, g=1, \ldots, G\}.$ 
 	Consider the following assumptions:
 	\begin{enumerate}
 		\item The $p$-values in each group are independent of the $p$-values in any other group, and the multiple testing procedure satisfies either of the following conditions:
 		\begin{itemize}
 			\item[1.1] The procedure is non-increasing and is self-consistent with respect to thresholds of form (\ref{non-adapt}) with identity shape function $\beta(x)=x.$
 			\item[1.2] The procedure is non-increasing and is self-consistent with thresholds of form (\ref{adapt1}), where $\hat{\pi}_0$ satisfies Condition \ref{cond-adapt}.
 		\end{itemize}  
 	\item The set of partial conjunction $p$-values $\{P_g^{u_g/n_g},\,g=1, \ldots, G\},$ satisfies the PRDS property on the subset of $p$-values corresponding to true partial conjunction hypotheses, i.e. those with indices in $\mathcal{G}_0$; the multiple testing procedure is non-increasing and is self-consistent with respect to thresholds of form (\ref{non-adapt}) with identity shape function $\beta(x)=x.$
 		\item The $p$-values across groups are arbitrarily dependent, and the multiple testing procedure is self-consistent with respect to thresholds 
 		of form (\ref{non-adapt})  with the shape function  of the form
 		\begin{align}\beta_{\nu}(r)=\int_{0}^{r}xd\nu(x),\label{beta}\end{align}
 		where $\nu$ is an arbitrary probability distribution on $(0,\infty).$ 
 	\end{enumerate} 
 If the assumptions of either of the three items above are satisfied, 
 then the multiple testing procedure guarantees $\text{FDR}_v^{PC}\leq \alpha.$
 \end{proposition}
 \begin{remark}
 For item 1 along with 1.1, as well as for items 2 and 3, one has a tighter bound for  $\text{FDR}_v^{PC}.$ Specifically, one has $$\text{FDR}_v^{PC}\leq \frac{\alpha}{m}\sum_{g \in \mathcal{G}_0}v_gw_g\leq \alpha.$$
 \end{remark}
This proposition follows immediately from the results of \cite{blanchard2008two} and \cite{blan2009}. Since the partial conjunction $p$-values are combinations of within-group elementary $p$-values, under the assumptions of item 1, the partial conjunction $p$-values are independent. Therefore, the result of item 1 along with item 1.1 follows from Propositions 2.7 and 3.3 of \cite{blanchard2008two}. The result of item 1 along with item 1.2 follows from Theorem 11 of \cite{blan2009}.  In the case of item 2, satisfy the PRDS property, while in the case of item 3, the partial conjunction $p$-values may be arbitrarily dependent. 
Therefore, items 2 and 3 follow from Propositions 3.6 and 3.7 in \cite{blanchard2008two} respectively, in conjunction with Proposition 2.7  in that paper.

As discussed in Section \ref{prel}, many multiple testing procedures with thresholds of form (\ref{non-adapt}) with $\beta(x)=x$ are non-increasing and self-consistent, including the BH procedure and its generalizations incorporating weights, captured by the doubly-weighted BH  procedure in \cite{blanchard2008two}. In addition,  several adaptive variants of the BH procedure satisfy the conditions of item 1.2, including the adaptive BH procedure with Storey's estimator (\ref{Storey}). 
Therefore, under independence across groups, all these procedures guarantee FDR control on partial conjunction hypotheses, if valid partial conjunction $p$-values are given as input to the procedure. However, as discussed in the introduction, when the $p$-values belonging to different groups are dependent, it may be not clear whether the partial conjunction $p$-values satisfy the PRDS property on the subset of true partial conjunction nulls, as required in item 2. In these cases, one may revert to replacing the identity shape function $\beta(r)=r$ by the shape function of form (\ref{beta}),  which gives FDR control guarantees under arbitrary dependence among the partial conjunction $p$-values, according to item 3. However, this replacement results in lower thresholds leading to loss of power, because $\beta_{\nu}(r)\leq r.$ For example, the adjustment of the BH procedure for addressing arbitrary dependence, suggested by \cite{BY01}, replaces the identity shape function by  $\beta_{\nu}(r)=r/(\sum_{j=1}^m1/j).$ In Section \ref{MT} it is shown that for several dependency structures among the elementary $p$-values, this conservative adjustment is not needed if certain methods for constructing partial conjunction $p$-values are used. 

\subsubsection{Results for partial conjunction p-values connected to multiple testing procedures}\label{MT}
  
  According to Corollary \ref{connection-cor},  the minimum of all the adjusted $p$-values of a multiple testing procedure is a valid global null $p$-value, which is non-decreasing in each elementary $p$-value, provided that the multiple testing procedure guarantees FDR control in the weak sense and is monotone in the sense of Remark \ref{monotone}. 
  In Section \ref{PC}, we showed the method of \cite{BH08} for obtaining a partial conjunction $p$-value based on a certain non-decreasing combining function for testing the global null. Taking these two results together, we may obtain partial conjunction $p$-values connected to multiple testing procedures. Examples are partial conjunction $p$-values based Simes', Hommel's, Bonferroni methods,  which are connected (in the above sense) to BH, BY, and Bonferroni procedures, respectively. Similarly, the partial conjunction $p$-value in (\ref{PC-Storey}) is connected to the adaptive BH method with Storey's estimator (\ref{Storey}) for the proportion of nulls. 
  
  Let us consider the case where for each $g\in\{1, \ldots, G\},$ the partial conjunction $p$-value $P^{u_g/n_g}$  is connected to a certain multiple testing procedure $\mathcal{M}_g$ which satisfies the natural monotonicity property given in Remark \ref{monotone}. Formally,
 \begin{align}P_g^{u_g/n_g}=\max\{P^{1/(n_g-u_g+1)}[A]: A\subseteq A_g, |A|=n_g-u_g+1\},\label{PCM}\end{align} where $P^{1/(n_g-u_g+1)}[A]$ 
 is the minimum adjusted $p$-value according to  $\mathcal{M}_g$  applied on the $p$-values with indices in $A.$ 
 The multiple testing procedures $\mathcal{M}_g$ may be different for different groups $g$, for example for certain groups one may use the partial conjunction $p$-value based on Simes' method, while for the other groups one may revert to Bonferroni's method. The next theorem addresses such partial conjunction $p$-values.

  \begin{thm}\label{min-adj-non-adapt}
 Let $\mathcal{M}$ be a self-consistent multiple testing procedure with respect to thresholds of form (\ref{non-adapt}) with identity shape function $\beta(x)=x.$	Consider a family of partial conjunction hypotheses $\{H_0^{u_g/n_g},\,\, g=1, \ldots, G\},$ associated with prior weights $\{w_g, \,\,g=1, \ldots, G\}$ and penalty weights $\{v_g, \,\,g=1, \ldots, G\},$ satisfying $\sum_{g=1}^Gw_gv_g=G.$ Let $\mathcal{M}_g, \, g=1, \ldots, G$ be multiple testing procedures satisfying the monotonicity property given in Remark \ref{monotone}. Assume that for each $g\in\{1, \ldots, G\},$ the partial conjunction $p$-value $P^{u_g/n_g}$ 
 is connected to  $\mathcal{M}_g$ in the sense of (\ref{PCM}). 
  In each of the following three cases, requiring additional assumptions on the  procedure $\mathcal{M},$ 
  the procedures $\mathcal{M}_g, \, g=1, \ldots, G,$ 
  and the dependency among the elementary $p$-values, 
  the  procedure $\mathcal{M}$ applied on $\{P_g^{u_g/n_g}, \,\,g=1, \ldots, G\}$ at level $\alpha$ guarantees
  $$\text{FDR}_v^{PC}\leq \frac{\alpha}{G}\sum_{g \in \mathcal{G}_0}v_gw_g\leq\alpha.$$ 
  \begin{enumerate} \item 
  	\begin{itemize} 
  	\item[(a)] The $p$-values $(p_1, \ldots, p_M)$ satisfy the overall positive dependence property, as defined in item D3 in Section \ref{setting}.
  	\item[(b)]The multiple testing procedure $\mathcal{M}$ is non-increasing.
  	\item[(c)] For each $g\in\{1, \ldots, G\},$ the multiple testing procedure $\mathcal{M}_g$ is self-consistent with respect to thresholds of form (\ref{non-adapt}) with prior and penalty weights equal to unity and identity shape function.
  			\end{itemize}
  	 	 \item \begin{itemize}
  	 	\item[(a)] The $p$-values $(p_1, \ldots, p_M)$ satisfy  conditional positive dependence across groups, and independence within each group, as defined in items D4 and D2 in Section \ref{setting}.
  	 	\item[(b)] The multiple testing procedure $\mathcal{M}$ is stable and concordant.
  	 	\item[(c)] For each $g\in\{1, \ldots, G\},$ the multiple testing procedure $\mathcal{M}_g$ is self-consistent with respect to thresholds of form (\ref{adapt1}), 
  	 	where $\hat{\pi}_0$ satisfies  Condition \ref{cond-adapt}. 
  	 \end{itemize} 
     	\item \begin{itemize}
   		\item[(a)] The $p$-values $(p_1, \ldots, p_M)$ satisfy positive dependence across groups, as defined in item D5 of Section \ref{setting}, and the dependence among the $p$-values within each group is unspecified.
   		\item[(b)]The multiple testing procedure $\mathcal{M}$ is stable and concordant.
   		\item[(c)] The partial conjunction $p$-values are connected to the Bonferroni procedure, i.e. for each $g\in\{1, \ldots, G\}$, $P_g^{u_g/n_g}=(n_g-u_g+1)p_{g(u_g)},$ where $p_{g(1)}\leq \ldots\leq p_{g(n_g)}$ is the sequence of ordered $p$-values for the hypotheses in group $g.$ 
   	\end{itemize}
  \end{enumerate}   	
  \end{thm}

The results of Theorem \ref{min-adj-non-adapt} rely on the sufficient conditions for FDR control obtained by \cite{blanchard2008two}, as well as an additional pair of such sufficient conditions, which we obtain using the techniques of \cite{blanchard2008two} (see Appendix \ref{appendix:suff}). Each pair of conditions addresses the multiple testing procedure and the dependency structure. The  condition addressing the dependency structure follows from Lemma \ref{ov-PRDS1}, which generalizes the  group-level superuniformity lemma of \cite{RetJ17}. 
The cases of independence across groups and arbitrary dependence across groups are not considered in this theorem, since they are covered by Proposition \ref{main-FDR}.  

The doubly-weighted BH procedure considered by (\cite{blanchard2008two}) is self-consistent with respect to thresholds of form (\ref{non-adapt}), is non-increasing, stable, and concordant. Therefore, for this procedure we obtain the following corollary:
\begin{corollary}\label{cor}
    The doubly-weighted BH procedure at level $\alpha$ applied on the partial conjunction $p$-values $\{P_g^{u_g/n_g},\, g=1, \ldots, G\}$ guarantees $\text{FDR}_v^{PC}\leq \alpha$ in each of the following three cases:
    \begin{itemize}
        \item[1]The $p$-values are positively dependent in the sense of item D3 in Section \ref{setting}, and the partial conjunction $p$-values are based on the combining method of Simes (\ref{Simes}), or on more conservative combining methods, such as Bonferroni (\ref{Bonf}) or Hommel (\ref{Hommel}).
  		\item[2] The $p$-values are independent within each group and are conditionally positively dependent across groups, in the sense of items D2 and D4 in Section \ref{setting}, and the partial conjunction $p$-values are computed using Simes-Storey function given in (\ref{PC-Storey}), i.e. they are connected in the sense of (\ref{PCM}) to the adaptive BH procedure with Storey's estimator for the proportion of nulls.
  		\item[3] The p-values are positively dependent across the groups, in the sense of item D5 in Section \ref{setting}, are arbitrarily dependent within the groups, and the partial conjunction $p$-values are based on Bonferroni's method.
    \end{itemize}
\end{corollary}
The result of item 1 with respect to global null hypotheses (i.e. the case where $u_g=1$ for $g=1, \ldots, G$), which are tested using Simes' $p$-values (\ref{Simes}) follows from the results of \cite{RetJ17}. Here it is generalized, addressing in addition partial conjunction hypotheses with $u_g>1$ and additional combining methods. It is easy to see that the dependency structure in item 2 satisfies the overall positive dependence requirement of item 1, 
so under the condition of item 2, the doubly-weighted BH procedure is valid when applied on Simes' partial conjunction $p$-values. However, the more strict  requirement of item 2 allows to replace these $p$-values by Simes-Storey partial conjunction $p$-values, which are more liberal for the groups which are enriched with signal. The result of item 2 is useful for the meta-analysis setting, see Section \ref{meta-analysis}. Finally, item 3 shows that when relaxing the independence within groups requirement in item 2, and requiring unconditional positive dependence across groups, the error control is still guaranteed when the doubly-weighted BH procedure is applied on the Bonferroni partial conjunction $p$-values. Theorem \ref{min-adj-non-adapt} shows that when the doubly-weighted BH procedure is replaced by a more conservative self-consistent procedure, possibly due to the structural constraints on the set of rejected partial conjunction hypotheses (see Section \ref{prel} for examples), we have the same  results as in Corollary \ref{cor} under very lenient assumptions on the procedure. Similarly, one may use more conservative partial conjunction tests, based on multiple testing procedures $\mathcal{M}_g$ which address certain structural constraints within groups.

 It may be of interest to consider the results of Theorem \ref{min-adj-non-adapt} for two extreme cases: $G=M$ and $G=1$. In the first case, each group consists of one hypothesis, so $n_g=u_g=1,$ and $\text{FDR}^{PC}_v$ reduces to the weighted FDR for the elementary hypotheses. The partial conjunction $p$-value  based on Simes'  method satisfies the assumption of part (c) in item 1, and in this extreme case,  for each group $g$ this partial conjunction $p$-value reduces to the elementary $p$-value of the single hypothesis in group $g$. 
 Both types of positive dependence across groups reduce to overall positive dependence in this case, therefore the result of Theorem \ref{min-adj-non-adapt} under the assumptions of item 1 follows from combining Proposition 3.6 and Proposition 2.7 of \cite{blanchard2008two}. 
 In  the second case, where $G=1,$ the entire set of $p$-values constitutes one group, with prior and penalty weight equal to unity. Consider Theorem \ref{min-adj-non-adapt} with  the BH procedure taking the role of $\mathcal{M}.$ In this extreme case, the BH procedure reduces to rejecting the single partial conjunction null when the given partial conjunction $p$-value is upper bounded by $\alpha.$ Therefore, $\text{FDR}_v^{PC}$ is simply the significance level of the partial conjunction test. Thus, Theorem \ref{min-adj-non-adapt} shows that in each of the cases 1, 2, and 3, a partial conjunction $p$-value satisfying assumption (c) is valid under the dependence assumption (a) within the single group. 
 Particularly, the partial conjunction $p$-values based on Simes', Simes-Storey, and Bonferroni methods, are valid under PRDS dependence, 
 independence, and arbitrary dependence,  respectively. These  are also known results, see Section \ref{sec-agg}. 
 
 Finally, we would like to note that Corollary \ref{cor} may be used for obtaining several results on Simes' $p$-value, which is the minimum adjusted $p$-value of the BH procedure. Obviously, each of the dependency structures considered in the corollary holds when we restrict the problem to any subset of groups, i.e. when we consider testing $\{H_0^{u_g/n_g},\,\, g\in \mathcal{G}_s\},$ where $\mathcal{G}_s\subseteq \{1, \ldots, G\}.$ Therefore, the BH procedure controls the FDR in each of the cases considered not only for the entire family of partial conjunction hypotheses, but also for any subset of those hypotheses. 
 Using this fact along with Corollary \ref{connection-cor}, we obtain the following result. 
 \begin{corollary}\label{cor-sim}
Let $\mathcal{G}_s\subseteq\{1, \ldots, G\}$ be a subset of group indices,  and consider applying Simes' combining function (\ref{Simes}) on the partial conjunction $p$-values in the set $\{P_g^{u_g/n_g}, \,\,g\in \mathcal{G}_s\}.$ The combined $p$-value is valid for testing the intersection of partial conjunction hypotheses $\cap_{g\in \mathcal{G}_s} H_0^{u_g/n_g}$ in each of the three cases considered in Corollary \ref{cor}. 
\end{corollary}
Obviously, the same result holds for the weighted variant of Simes' $p$-value, which is the minimum adjusted $p$-value of the doubly-weighted BH procedure. The fact that  Simes' $p$-value is valid when it receives as input Simes' global null $p$-values and all the elementary $p$-values are PRDS has been proven by \cite{RetJ17} (see item (d) of their group-level superuniformity lemma). Corollary 
\ref{cor-sim} generalizes this result, addressing general Simes' partial conjunction $p$-values, as well as other dependency structures and combining methods.  Consider a procedure based on the closed testing principle of \cite{marcus1976closed}, which
considers the family $\{H_0^{u_g/n_g}, g=1, \ldots, G\}$ as the family of elementary hypotheses, and tests the intersections of these hypotheses using Simes' test applied on the corresponding partial conjunction $p$-values. Based on 
Corollary \ref{cor-sim}, we obtain that in each of the three cases of Corollary \ref{cor}, all these Simes' tests are valid, therefore this procedure guarantees FWER control on the partial conjunction hypotheses, as well as on the entire set of all their intersections.
In fact, it reduces to Hommel's procedure (\cite{hommel1988stagewise}) applied on partial conjunction $p$-values. We conclude that Hommel's procedure, as well as its shortcuts (see, e.g., \cite{hochberg1988sharper}, \cite{meijer2019hommel}, \cite{goeman2011multiple}, \cite{goeman2019simultaneous}) are valid for FWER control on the partial conjunction hypotheses in cases 1--3 of Corollary \ref{cor}. Moreover, Goeman and Solari (\cite{goeman2011multiple}) showed that closed testing with valid local tests can be used for controlling the false discovery proportions (FDP) over all subsets, and suggested using this approach for more flexible multiple testing. In our setting, this approach gives the flexibility to select a subset of groups  post-hoc, and to estimate the number of false partial conjunction hypotheses for those selected groups. We discuss this goal and the implications of Corollary \ref{cor-sim} in more detail in Section \ref{discussion}. 

\section{The meta-analysis setting}\label{meta-analysis}
 \subsection{Preliminaries}\label{setting-metaanalysis}
 We consider the setting where the set of $m$ features are tested in $n\geq 2$ independent studies. The corresponding hypotheses can be arranged in a matrix with $m$ rows and $n$ columns, where the $(i,j)$th entry of the matrix is $H_{ij},$  the null hypothesis for feature $i$ in study $j.$ Consider the corresponding $m\times n$ matrix of $p$-values for these hypotheses, $\mathbf{p},$ where $p_{ij}$ is the $p$-value for $H_{ij},$ $i=1, \ldots, m;$ $j=1, \ldots, n.$ The hypotheses belonging to the same row are those corresponding to the same feature in different studies.
 
 For feature $i\in\{1, \ldots, m\},$ let $k(i)$ be the number of studies where feature $i$ has an effect, i.e. the number of false null hypotheses in the set $\{H_{ij}, j=1, \ldots, n\}.$ For $u\in\{1, \ldots,n\},$ consider the family of partial conjunction nulls $\{H_{i}^{u/n}, i=1, \ldots, m\},$ where $H_i^{u/n}$ states that feature $i$ has an effect in less than $u$ studies, i.e. $H_i^{u/n}: k(i)<u.$ For $u=1,$ identifying false null hypotheses in this family corresponds to identifying features with an effect in at least one study, which corresponds to the typical goal of meta-analysis. Testing this family for $u\geq 2$ corresponds to identifying features with an effect in at least $u$ studies, which is considered as the goal of replicability analysis (see, e.g., \cite{BHY09}, \cite{heller2014deciding}, \cite{bogomolov2018assessing}, \cite{wang2016detecting}, \cite{hoang2021combining}).  
 
 Denote by $\bm{p}_{\cdot j}$ the vector of $m$ $p$-values in column $j,$ and by $\mathbf{p}_{\cdot(-j)}$ the $m\times (n-1)$ matrix which is obtained from $\mathbf{p}$ by excluding column $j.$ Similarly, $\bm{p}_{i\cdot}$ and $\mathbf{p}_{(-i)\cdot}$ denote the vector of $n$ $p$-values in row $i,$ and the $(m-1)\times n $ matrix which is obtained from $\mathbf{p}$ by excluding row $i,$ respectively. We denote by $\bm p,$ $\bm{p}_{\cdot(-j)},$ and $\bm p_{(-i)\cdot}$ the vectors obtained by stacking the rows of $\mathbf{p},$ $\mathbf{p}_{\cdot(-j)},$ and $\mathbf p_{(-i)\cdot},$ respectively.
 
\subsection{Multiple testing of partial conjunction hypotheses in the meta-analysis setting}\label{meta-mult}
Similarly to Section \ref{mult-PC}, 
we address testing the family $\{H_{i}^{u/n}, i=1, \ldots, m\}$ with (weighted) FDR control. 
A typical meta-analysis addresses independent studies, so we assume independence within each row: for each $j\in\{1, \ldots, n\},$ the vectors $\bm{p}_{\cdot j}$ and  $\bm{p}_{\cdot (-j)}$ are independent.  Thus, for obtaining valid partial conjunction $p$-values, one may use any combining method which is valid under independence, such as Stouffer's, Fisher's, or Simes-Storey (\ref{PC-Storey}) methods. The dependency within the columns, in combination with the chosen partial conjunction $p$-values, guides the choice of
the multiple testing procedure, which should guarantee the desired error rate control for the induced dependencies among those combined $p$-values. For the results below, 
the following dependencies are considered within each study $j\in\{1, \ldots, n\},$ in conjunction with the assumption of independence across studies: 
\begin{itemize}
	\item[M1] Positive dependence: the vector of $p$-values for the hypotheses in study $j$, $\bm{p}_{\cdot j},$  satisfies the PRDS property on the subset of true null hypotheses.
	\item[M2] Independence: the $p$-values within $\bm{p}_{\cdot j}$ are  independent.
	\item[M3] Arbitrary dependence: the dependence among $\{p_{ij}, i=1, \ldots, m\}$ is unspecified. 
\end{itemize}
To obtain an analogy to the dependency structures in Section \ref{mult-PC}, it is useful to note that in the current matrix setting, the hypotheses may be viewed as  divided into $m$ groups, where the hypotheses within each row of the matrix form a group. As in Section \ref{mult-PC}, a partial conjunction hypothesis is considered for
each group (corresponding to a specific feature), and the goal is (weighted) FDR control for the tested family of $m$ partial conjunction hypotheses. The independence across the studies translates to independence of $p$-values within each group. Assuming in addition independence within each study (as defined in item M2), we obtain independence across the groups, as defined in item D1 in Section \ref{mult-PC}. 
Similarly,  arbitrary dependence within each study translates to arbitrary dependence across groups. Therefore, the theoretical results for the dependencies in items M2 and M3 are covered by items 1 and 3 of Proposition \ref{main-FDR}, respectively. Let us consider the dependency structure in item M1, where the $p$-values are positively dependent within each study, and are independent across the studies. We show in the appendix that this dependency structure satisfies  conditional positive dependence across groups and independence within groups, defined in items D4 and D2 in Section \ref{setting}, respectively, as well as overall positive dependence, defined in item D3. Therefore, the results of Theorem \ref{min-adj-non-adapt} hold for the dependency structure in M1. The following theorem states this result, and an additional result addressing Fisher's and Stouffer's methods for obtaining partial conjunction $p$-values.

\begin{thm}\label{thm-meta}
	Consider the meta-analysis setting, where the interest lies in testing the family $\{H_i^{u/n}, \,i=1, \ldots, m\}$ for a specific $u\in\{1, \ldots, n\}.$ Assume that this family is associated with prior weights $\{w_i, \,i=1, \ldots, m\},$ and penalty weights $\{v_i, \,i=1, \ldots, m\},$ satisfying $\sum_{i=1}^m w_iv_i=m.$ Let $\mathcal{M}$ be a non-increasing  multiple testing procedure, which is self-consistent with respect to thresholds of form (\ref{non-adapt}) with identity shape function $\beta(x)=x.$
	Assume independence of $p$-values across the studies, meaning that for each $j\in\{1, \ldots,n\},$ the vectors $\bm{p}_{\cdot j}$ and   $\bm{p}_{\cdot (-j)}$ are independent.
	In addition, assume positive dependence within each study, in the sense of item M1. 
	In each of the following cases, procedure $\mathcal{M}$ applied on $\{P_i^{u/n},\, i=1, \ldots, m\}$  at level $\alpha$ guarantees
  $$\text{FDR}_v^{PC}\leq \frac{\alpha}{m}\sum_{i \in M_0}v_iw_i\leq\alpha,$$
	where $M_0=\{i\in\{1, \ldots, m\}: H_i^{u/n} \text{ is true}\},$  the index set of features for which the partial conjunction null is true.
	\begin{enumerate}
	\item The partial conjunction $p$-values are connected to a multiple testing procedure $\mathcal{M}_g,$ in the sense of (\ref{PCM}), where $\mathcal{M}_g$ and $\mathcal{M}$ either satisfy the assumptions of items 1(b) and 1(c) of Theorem \ref{min-adj-non-adapt}, respectively, or satisfy the assumptions of 2(b) and 2(c) of Theorem \ref{min-adj-non-adapt}, respectively.  
	\item  The partial conjunction $p$-values are based on Fisher's or Stouffer's methods, and the elementary $p$-values for true null hypotheses are $U(0,1)$ random variables.
	\end{enumerate}
\end{thm}

\begin{remark}
All our theoretical results are obtained under the assumption that the elementary $p$-values are valid, i.e. for each true null elementary hypothesis, the $p$-value is either $U(0,1)$ random variable, or is stochastically larger than this variable. For the result of item 2 of Theorem \ref{thm-meta}, we make a stronger assumption, which  holds when all the test statistics are continuous. 
\end{remark}

As a corollary, we obtain that under independence across studies and positive dependence within each study, the doubly-weighted BH procedure guarantees weighted FDR control on the family $\{H_i^{u/n}, i=1, \ldots, m\},$ when the $p$-values are based on Simes', Simes-Storey (\ref{PC-Storey}), Fisher's, or Stouffer's methods, and the test statistics are continuous.
The result of item 2 is a generalization of the result of \cite{BH08}, showing that in the matrix setting, where the $p$-values within each column are positively dependent, and the $p$-values across the columns are independent, the BH procedure guarantees FDR control when applied on partial conjunction $p$-values based on Stouffer's or Fisher's methods. Item 2 shows that this result remains true for the weighted variants of the BH procedure, as well as other self-consistent procedures, possibly incorporating structural constraints, such as the procedures of \cite{BetS17a} and \cite{filtering}. This result is based on  Lemma \ref{lemma-meta},  which may be of independent interest.
\subsection{Assessing the replicability extent for each feature}\label{proc-rep}
As discussed above, identifying features with replicated signals can be made by testing $\{H_i^{u/n}, i=1, \ldots, m\}$ for $u\geq 2.$  However, as noted by \cite{BHY09}, it may be not clear what value of $u$ should be chosen to establish replicability. 
For example, if $n>2,$ the replicability strength for features with an effect in all studies, i.e. those with $k(i)=n,$ is stronger than for features with an effect in only two studies. Thus, it may be of interest to assess the replicability extent for certain features.
For this purpose, Benjamini et al. \cite{BHY09} developed a procedure, which selects the features with evidence for having an effect in at least one study by applying the BH procedure on the global null $p$-values $\{P_i^{1/n}, i=1, \ldots, m\},$ and estimates, for each selected feature $i,$ a lower bound $\hat{k}(i)$ for $k(i),$ the number of studies where feature $i$ has an effect. This procedure  is based on sequential testing of partial conjunction hypotheses for the selected features. Benjamini et al. \cite{BHY09} proved that it controls the expected proportion of features with false replicability claims, i.e. $\hat{k}(i)>k(i),$ out of all the selected features, under independence within each study. 

In this section an extension of the procedure of \cite{BHY09} is proposed, allowing general selection rules based on the entire matrix $\mathbf{p}_{m\times n}$ for selecting the features for which assessment of replicability strength will be made, as well as incorporating prior and/or penalty weights for the features.
Allowing general selection rules seems to be valuable in practice: for example, a researcher may wish to assess the replicability of findings of a specific study, addressing the BH procedure on one column of the $p$-value matrix as the selection rule. Alternatively, one may be interested in establishing replicability for features with global null $p$-values smaller than a pre-specified threshold, such as Bonferroni's threshold. 
Prior weights may be used for incorporating prior knowledge regarding the features, both in the selection step and in assessing replicability. Penalty weights may be incorporated in the error rate of \cite{BHY09} for assigning different prices for erroneous replicability claims. 
Based on the theoretical results we developed in previous sections, we also extend the theoretical results of \cite{BHY09}, addressing positive and arbitrary dependencies within the studies, in addition to independence.

Assume we are given a vector of penalty weights $(v_1, \ldots, v_m)$ and a vector of prior weights $(w_1, \ldots, w_m)$ satisfying $\sum_{i=1}^m v_iw_i=m,$ as well as a selection rule $\mathcal{S},$ which receives as input the matrix $\mathbf{p}$ and outputs the indices of selected features $\mathcal{S}=\mathcal{S}(\mathbf{p})\subseteq\{1, \ldots, m\}.$ The target error rate of the generalized procedure is 
	\begin{align}\EE{\frac{\sum_{i\in \mathcal{S}}v_i\ind(\hat{k}(i)>k(i))}{|\mathcal{S}|_v,}}, \label{control}\end{align}
 This error rate reduces to the target error rate of \cite{BHY09} when all the penalty weights are equal to unity. Given a shape function $\beta:\mathbb R^+\rightarrow \mathbb R^+,$ the generalized procedure targeting control of (\ref{control}) at level $q$ is given  by the following two-step algorithm.
	\begin{enumerate}
		\item Apply a selection rule $\mathcal{S},$ which receives the matrix of $p$-values $\mathbf{p}$ and outputs the indices of selected features, $\mathcal{S}\subseteq \{1, \ldots, m\}.$ 
		\item For each  selected feature $i\in \mathcal{S}$, test sequentially the partial conjunction hypotheses with $u=1, \ldots, n$ at level $w_i\beta(|\mathcal{S}|_v)q/m,$ 
	in order to find 
	$$\hat{k}(i)=\max\left\{u: \max\{P_i^{1/n}, P_i^{2/n}, \ldots, P_i^{u/n}\}\leq w_i\beta(|\mathcal{S}|_v)q/m\right\},$$
	and claim that $k(i)\geq \hat{k}(i),$ i.e. feature $i$ has an effect in at least $\hat{k}(i)$ studies. The maximum of an empty set is defined as 0. 
   	\end{enumerate}
When the selection rule is the BH procedure applied on the global null $p$-values, and the prior and unity weights are equal to unity, the above procedure reduces to the procedure of \cite{BHY09}. This procedure guarantees that  each selected feature $i$ has a non-trivial lower bound, i.e. $\hat{k}(i)>0,$ provided that the global null $p$-values which are used in Step 1 are computed using the same method as the partial conjunction $p$-values used in Step 2. The same property holds for the generalized procedure, when the BH selection rule is replaced by any multiple testing procedure applied on global null $p$-values, provided that the procedure is self-consistent with respect to thresholds of form (\ref{non-adapt}), with the same shape function and weights as those used in Step 2. 
For theoretical results, we consider the following assumptions on the selection rule in Step 1.
\begin{itemize}
	\item[C1] The selection rule $\mathcal{S}$ is \textit{non-increasing}, in the sense that $|\mathcal{S}(\mathbf{p}^{(1)})|_{\bm v}\geq|\mathcal{S}(\mathbf{p}^{(2)})|_{\bm v},$ if  $\mathbf{p}^{(1)}$ and $\mathbf{p}^{(2)}$ are $m\times n$ matrices satisfying $p^{(1)}_{ij}\leq p^{(2)}_{ij},$ for each $i\in\{1, \ldots, m\}$ and $j\in\{1, \ldots,n\}.$ 
	\item [C2] The selection rule $\mathcal{S}$
	is \textit{stable} and \textit{concordant} with respect to rows of $\mathbf{p}_{m\times n},$ i.e. the following two conditions are satisfied:
	\begin{itemize}
	    \item[(a)] For each $i\in\mathcal{S}(\mathbf{p}),$ fixing all $p$-values in $\mathbf{p}_{(-i)\cdot},$ and changing the $p$-values in row $i$ so that feature $i$ is still selected, will not change the set $\mathcal{S}.$  Let $\mathcal{S}^{(-i)\cdot}=\mathcal{S}^{(-i)\cdot}(\mathbf{p}_{(-i)\cdot})$ be the set of selected features when the $p$-values not corresponding to feature $i$ are defined by $\mathbf{p}_{(-i)\cdot},$ and the $n$ $p$-values for feature $i$ change as long as feature $i$ is selected.
	    \item[(b)] For any two $(m-1)\times n$ matrices  $\mathbf{p}^{(1)}_{(-i)\cdot}$ and $\mathbf{p}^{(2)}_{(-i)\cdot}$ satisfying $\mathbf{p}^{(1)}_{(-i)\cdot}\leq \mathbf{p}^{(2)}_{(-i)\cdot},$ where the inequality is understood coordinate-wise, it holds that $|\mathcal{S}^{(-i)\cdot}(\mathbf{p}^{(1)}_{(-i)\cdot})|_{\bm v}\geq |\mathcal{S}^{(-i)\cdot}(\mathbf{p}^{(2)}_{(-i)\cdot})|_{\bm v}.$ 
	\end{itemize}
\end{itemize}
When $n=1$ and the features are selected if the corresponding hypotheses are rejected by a given multiple testing procedure, each of the above properties 
reduces to the corresponding property of the given multiple testing procedure, as defined in Section \ref{prel}. 
Moreover, for the general case where $n>1,$ assume that the selection of features is made by computing a certain combination of $p$-values within each row (e.g. the global null $p$-values for each feature), and applying a multiple testing procedure on these $m$ combined $p$-values. Assume that the combining function is non-decreasing in each coordinate, as it happens for common global null combining functions. Then the selection rule is non-increasing if the corresponding multiple testing procedure is non-increasing, and the same is true with respect to stability and concordance. Therefore, the selection rule of \cite{BHY09}, which selects the features by applying BH on global null $p$-values, satisfies Conditions C1 and C2 provided that the global null $p$-values are non-decreasing in each elementary $p$-value. Moreovoer, if BH is replaced by its weighted variant or another self-consistent procedure with respect to thresholds of form (\ref{non-adapt}), these conditions are still satisfied. Finally, let us consider selecting the features by applying a multiple testing procedure on the $p$-values of a specific study $j.$ Since for each row $i,$ the combining function $f(p_{i1}, \ldots, p_{in})=p_{ij}$ is non-decreasing in each $p$-value, any non-increasing, stable and concordant multiple testing procedure, such as those given above, will yield a selection rule with the same properties. 
\begin{thm}\label{meta-seq}
Assume independence of $p$-values across the studies, meaning that for each $j\in\{1, \ldots,n\},$ the vectors $\bm{p}_{\cdot j}$ and $\bm{p}_{\cdot (-j)}$ are  independent. 
The generalized procedure consisting of two steps above guarantees 
	\begin{align*}\EE{\frac{\sum_{i\in \mathcal{S}}v_i\ind(\hat{k}(i)>k(i))}{|\mathcal{S}|_v}}\leq q,\end{align*}
i.e. it controls the expected weighted proportion of selected features for which false replicability claims were made, in each of the following cases. 
	\begin{enumerate}
		\item The shape function $\beta$ is the identity function $\beta(x)=x,$ the $p$-values within each study are positively dependent in the sense of item M1 in Section \ref{meta-mult}, and either of the following conditions hold:
		\begin{enumerate}
		    \item The selection rule $\mathcal{S}$ is non-increasing, and the  partial conjunction $p$-values are either based on Fisher's or Stouffer's methods (assuming the condition of item 2 of Theorem \ref{thm-meta}), or are connected in the sense of (\ref{PCM}) to a multiple testing procedure which is non-increasing and self-consistent with respect to thresholds of form $\Delta(i, r)=r\alpha/n,$ such as those based on Simes' method. 
		    
		    \item The selection rule satisfies condition C2, and the partial conjunction $p$-values are connected in the sense of (\ref{PCM}) to a multiple testing procedure which is self-consistent with respect to thresholds of the form $\Delta(i,r)=r\alpha/n\hat{\pi}_0^{-1},$ with the estimator $\hat{\pi}_0$ satisfying Condition \ref{cond-adapt} with $m$ replaced by $n$, such as those based on Simes-Storey method (\ref{PC-Storey}).
		\end{enumerate}
		\item The shape function $\beta$ is the identity function $\beta(x)=x,$  the $p$-values within each study are independent, the selection rule $\mathcal{S}$ is stable, and the partial conjunction $p$-values may be arbitrary, as long as they are valid under independence.
		\item The  shape function $\beta$ is of form (\ref{beta}), e.g. $\beta(x)=x/(\sum_{i=1}^m1/i),$ and the method for obtaining partial conjunction $p$-values is valid under independence. The dependence within each study and the selection rule $\mathcal{S}$ may be arbitrary. 
	\end{enumerate}
\end{thm}
As mentioned above, the procedure of \cite{BHY09} is a special case of the procedure considered in Theorem \ref{meta-seq}, and since the selection rule considered in \cite{BHY09} is non-increasing, the theoretical result of \cite{BHY09} is covered by item 2. 
Theorem \ref{meta-seq} generalizes the result of \cite{BHY09}, showing that it remains to hold under positive dependence within each study for several  choices of partial conjunction $p$-values, such as those based on Fisher's, Stouffer's, Simes' and Simes-Storey (\ref{PC-Storey}) methods. Moreover, the BH procedure may be replaced by other selection rules, and one may incorporate prior and penalty weights, provided that the above conditions hold. Finally, according to item 3, one may address arbitrary selection rules and arbitrary dependencies within the studies by replacing the threshold $w_i|\mathcal{S}|_vq/m$ in Step 2 by a more conservative threshold, e.g. $w_i|\mathcal{S}|_vq/(m\sum_{i=1}^m1/i).$


\section{Technical remarks and key lemmas}\label{sec-lemmas}
The main results in this paper are based on the following lemmas, in conjunction with Proposition 2.7 in \cite{blanchard2008two}, which gives two sufficient conditions for (weighted) FDR control, and a similar proposition addressing stable procedures (see Appendix \ref{appendix:suff}). The first condition
is on the multiple testing procedure, requiring self-consistency with respect to thresholds of form (\ref{non-adapt}). The second condition addresses the dependency structure among the $p$-values,
requiring that for each true null hypothesis, its $p$-value and $|\mathcal{R}(\bm{p})|_v$ (replaced by another quantity for stable procedures), satisfy the dependency control condition, defined below. 
\begin{definition}[Dependency control condition, \cite{blanchard2008two}]
Let $\beta: \mathbb{R}^+\rightarrow \mathbb{R}^+$ be a non-decreasing function. A couple $(U, V)$ of non-negative real random variables satisfy the dependency control condition with shape function $\beta$ if the following inequalities hold:
$$\forall c>0,\,\,\,\EE{\frac{\ind(U\leq c\beta(V))}{V}}\leq c.$$
\end{definition}
The lemmas below yield that the second sufficient condition  for (weighted) FDR control is satisfied under the conditions of Theorems \ref{min-adj-non-adapt} and \ref{thm-meta}. Lemma \ref{lemma-meta} is used for proving Theorem \ref{meta-seq}. Given a function $g: [0, 1]^k\rightarrow [0,\infty),$  we say that $g$ is non-increasing if for any two vectors $\bm{x}_1, \bm{x}_2\in[0,1]^k$ satisfying $\bm{x}_1\leq \bm{x}_2$ where the inequality is understood coordinate-wise, it holds that $g(\bm{x}_1)\geq g(\bm{x}_2).$ 
\begin{lem}\label{ov-PRDS1}
Consider the setting of Section \ref{setting}, where  $\bm{p}$ is the vector  of all the elementary $p$-values, and $\bm{p}^{-g}$ is the vector of all the $p$-values excluding those belonging to group $g,$ for $g=1, \ldots, G.$ Let $g\in \mathcal{G}_0,$ so the partial conjunction null for group $g,$ $H_0^{u_g/n_g},$ is true. Assume that $P_g^{u_g/n_g}$ is connected to a multiple testing procedure $\mathcal{M}_g$ in the sense of (\ref{PCM}). Let $f:[0,1]^M\rightarrow [0, \infty)$ and $h:[0,1]^{M-n_g}\rightarrow [0, \infty)$ be two 
non-increasing functions. 
\begin{enumerate}
	\item 	Assume that the elementary $p$-values satisfy the  overall positive dependence conditions, as defined in item D3 of Section \ref{setting}. 
	Assume that $\mathcal{M}_g$ satisfies condition (c) of item 1 of Theorem \ref{min-adj-non-adapt}, which holds, for example, if $P_g^{u_g/n_g}$ is based on Simes' method.  Then the pair 
	$(P_g^{u_g/n_g}, f(\bm{p}))$ satisfies the dependency control condition with the identity shape function $\beta(x)=x.$ 
	\item Assume that the elementary $p$-values satisfy conditional positive dependence across groups and independence within each group, in the sense of items D4 and D2 in Section \ref{setting}. Assume that  $\mathcal{M}_g$ satisfies condition (c) of item 2 of Theorem \ref{min-adj-non-adapt}, which holds, for example, if $P_g^{u_g/n_g}$ is based on Simes-Storey method (\ref{PC-Storey}).    Then the pair $(P_g^{u_g/n_g}, h(\bm{p}^{-g}))$ satisfies the dependency control condition with the identity shape function $\beta(x)=x.$ 
	\item Assume that the elementary $p$-values satisfy positive dependence across groups in the sense of item D5 in Section \ref{setting}. Assume that $\mathcal{M}_g$ is the Bonferroni procedure, i.e. $P_g^{u_g/n_g}=(n_g-u_g+1)p_{g(u_g)},$ where $p_{g(1)}\leq \ldots\leq p_{g(n_g)}$ is the sequence of ordered $p$-values for the hypotheses in group $g.$ 
	Then, under any dependency within each group, the pair $(P_g^{u_g/n_g}, h(\bm p^{-g}))$ satisfies the dependency control condition with the identity shape function $\beta(x)=x.$
\end{enumerate}	
\end{lem}
Let us fix a certain partial conjunction $p$-value $P^{u_g/n_g}.$ The result that the dependency control condition holds for $(P_g^{u_g/n_g}, f(\bm{p}))$ with the identity shape function $\beta(x)=x$ for any non-increasing function $f:[0,1]^M\rightarrow [0, \infty)$  is equivalent to the result that the  inequality
\begin{align}\EE{\frac{\ind(P_g^{u_g/n_g}\leq f(\bm p))}{f(\bm p)}}\leq 1
	\label{itemRamdas1}\end{align}
	is satisfied for any such function $f.$ The equivalence is obtained as follows. Assume that (\ref{itemRamdas1}) is satisfied for any non-increasing function $f:[0,1]^M\rightarrow [0, \infty).$ Given such function $f$ and a constant $c>0,$  using the inequality in (\ref{itemRamdas1}) for $\tilde{f}=cf,$ which is also a non-increasing function, we obtain 
	\begin{align}\EE{\frac{\ind(P_g^{u_g/n_g}\leq cf(\bm{p}))}{f(\bm{p})}}\leq c
	\label{itemRamdas3},\end{align}
	which implies that the pair $(P_g^{u_g/n_g}, f(\bm{p}))$ satisfies the dependency control condition. On the other hand,  if the inequality in  (\ref{itemRamdas3}) holds for any non-increasing function $f:[0,1]^M\rightarrow[0, \infty)$ and for any $c>0,$ then it holds in particular for $c=1,$ implying (\ref{itemRamdas1}) for any such function $f.$
	The inequality  (\ref{itemRamdas1}) 
	is similar in form to the inequalities in the group-level superuniformity lemma of \cite{RetJ17}. Moreover, based on the equivalence we showed above, we obtain that item 1 of Lemma \ref{ov-PRDS1} for the case where $u_g=1$ and $P^{1/n_g}$ is based on Simes' method reduces to item (b) of the group-level superuniformity lemma of \cite{RetJ17}. Therefore, Lemma \ref{ov-PRDS1} may be viewed as complementing the group-level superuniformity lemma of \cite{RetJ17}, extending the results for Simes' global null $p$-values to general partial conjunction $p$-values, and addressing additional combining methods and dependency structures. Since the latter lemma was useful for proving FDR control results for some methods addressing groups of hypotheses, such as  $p$-filter \cite{RetJ17}, DAGGER \cite{ramdas2019sequential},  TreeBH \cite{bogomolov2020hypotheses}, and Focused BH \cite{filtering}, Lemma \ref{ov-PRDS1} may be useful for extending those results. 
	
	The following lemma addresses the meta-analysis setting given in Section \ref{setting-metaanalysis}, and is useful for obtaining  theoretical results  addressing positive dependence within each study. The results addressing independence and arbitrary dependence follow directly from the results of \cite{blanchard2008two}.  
\begin{lem}\label{lemma-meta}
	Consider the meta-analysis setting given in Section \ref{setting-metaanalysis}. Assume independence of $p$-values across the studies, meaning that for each $j\in\{1, \ldots,n\},$ the vectors $\bm{p}_{\cdot j}$ and $\bm{p}_{\cdot (-j)}$ are independent. In addition, assume positive dependence within each study, in the sense of item M1 in Section \ref{meta-mult}.
	 Let $i$ be a feature for which the partial conjunction null $H_i^{u/n}$ is true. 
	Let $f:[0,1]^{m\times n}\rightarrow [0, \infty)$ and $h:[0,1]^{(m-1)\times n}\rightarrow [0, \infty)$ be two non-increasing functions. 
	\begin{enumerate}
	\item  
	Assume that either of the following conditions holds: (i) $P_i^{u/n}$ is  based on Fisher's or Stouffer's methods, and the elementary $p$-values for true null hypotheses are $U(0,1)$ random variables, or (ii) $P_i^{u/n}$ is connected in the sense of (\ref{PCM}) to a multiple testing procedure , which satisfies condition (c) of item 1 of Theorem \ref{min-adj-non-adapt}, for example $P_i^{u/n}$ is based on Simes' method.  Then the pair
$(P_i^{u/n}, f(\bm{p}))$ satisfies the dependency control condition with  the identity shape function $\beta(x)=x.$
\item Assume that $P_i^{u/n}$ is connected in the sense of (\ref{PCM}) to a multiple testing procedure which satisfies condition (c) of item 2 of Theorem \ref{min-adj-non-adapt}, 
for example $P_i^{u/n}$ is based on Simes-Storey method (\ref{PC-Storey}).  Then the pair
$(P_i^{u/n}, h(\bm{p}_{(-i)\cdot})$ satisfies the dependency control condition with the identity shape function $\beta(x)=x.$
	\end{enumerate}
\end{lem}
The result regarding Fisher's and Stouffer's method is shown based on the proof techniques of \cite{BH08}. Similarly to \cite{BH08}, we rely on the theorem of Efron \cite{efron1965increasing}. The remaining results follow from the 
fact that the hypotheses corresponding to the same feature in different studies may be viewed as a group in the setting of Section \ref{setting},
and the dependency structure of Lemma \ref{lemma-meta} with respect to this group structure satisfies the conditions of items 1 and 2 of Lemma \ref{ov-PRDS1}.  
Similarly to the results of Lemma \ref{ov-PRDS1},  the results of Lemma \ref{lemma-meta} may be given in the form of (\ref{itemRamdas1}), and may also be viewed as generalizing the group-level superuniformity lemma of \cite{RetJ17}.
\section{Discussion}\label{discussion}
We obtained sufficient conditions for FDR control on partial conjunction hypotheses, addressing different dependency structures across and within the groups. For a given division into groups, we addressed two goals: (1) testing 
a pre-specified family of partial conjunction hypotheses $\{H_0^{u_g/n_g},\,\, g=1,\ldots, G\}$ with FDR control, and (2) selecting a subset of interesting groups via a pre-specified selection rule, and estimating the  number of false null hypotheses in each selected group, so that the expected  fraction of groups with over-estimated number of false null hypotheses, among all the selected groups, is bounded by a pre-specified level $\alpha.$  The latter goal was addressed in the meta-analysis setting, however, as we show in Section \ref{meta-mult}, this setting can be considered as a group setting of Section \ref{mult-PC}. 
Therefore,  these two goals can be unified as follows. 
For a single set of hypotheses $\mathcal{R}\subseteq \{1, \ldots, G\},$ defined by an a-priori fixed selection rule, it is required that
\begin{align}\EE{\frac{\sum_{g\in \mathcal{R}}v_g\ind(|A_g\cap M_1|< r_g)}{|\mathcal{R}|_v}}\leq\alpha,\label{fdr-gen}\end{align}
where $M_1\subseteq\{1, \ldots, M\}$ is the index set of all the false null hypotheses, $r_g=u_g$  for the procedures in Section \ref{mult-PC}, and $r_g$ is the estimator for the number of false null hypotheses in $A_g,$ which is an additional output of the procedure in Section \ref{proc-rep}. 
In the setting of Section \ref{mult-PC}, the selection rule is defined by the multiple testing procedure which is applied on the partial conjunction $p$-values, while in the setting of Section \ref{proc-rep}, the selection rule may be more general, however it should be fixed in advance. 

Goeman and Solari (\cite{goeman2011multiple}) suggested a flexible approach to multiple testing, where a researcher may select any subset of hypotheses $S\subseteq\{1, \ldots, M\}$, based on the data and possibly prior knowledge, and obtain a lower confidence bound $d^M(S)$ for the number of false null hypotheses in this set. For valid post-hoc inference, the lower bounds are required to satisfy the following simultaneity property:
\begin{align}\PP{d^M(S)\leq |S\cap M_1|\text{ for all }S\subseteq \{1, \ldots,M\}}\geq 1-\alpha.\label{lower}\end{align}
These lower bounds can be used for obtaining  simultaneous upper bounds for the false discovery proportions (FDP) for all the groups, therefore procedures for obtaining these bounds are often referred to as FDP-controlling procedures. In our group setting, the approach of Goeman and Solari (\cite{goeman2011multiple}) may be translated to having the flexibility to select any subset of groups $K\subseteq\{1, \ldots, G\},$ and obtain a lower bound on the number of false null hypotheses within each selected group $g\in K,$ rather than obtaining lower bounds only for one set $\mathcal{R},$ outputted by a fixed in advance selection rule. This flexibility can be obtained by replacing the requirement in (\ref{fdr-gen}) by the following:
\begin{align}
\PP{r_g\leq |A_g\cap M_1|\text{ for all }g\subseteq \{1, \ldots,G\}}\geq 1-\alpha,\label{gen}\end{align}
where $r_g$ are random bounds, and $r_g\in\{0, u_g\}$ in the setting of Section \ref{mult-PC}, while $r_g\in \{0,1,\ldots, n_g\}$ in the setting of Section \ref{proc-rep}. It is easy to see that (\ref{gen}) is a more strict requirement than our requirement (\ref{fdr-gen}), i.e. (\ref{gen}) implies (\ref{fdr-gen}). Indeed, (\ref{gen}) is equivalent to
\begin{align}
\EE{\ind\left\{\sum_{g=1}^G\ind(|A_g\cap M_1|<r_g)>0\right\}}\leq \alpha,\label{gen-1}\end{align}
which implies (\ref{fdr-gen}).
%
Obviously, given simultaneous lower bounds for all the subsets, i.e. $d^M(S)$ satisfying (\ref{lower}), we obtain lower bounds $r_g$ satisfying (\ref{gen}) for the given set of groups with index set $g=1, \ldots, G$ in the setting of Section \ref{proc-rep}. In the setting of Section \ref{mult-PC}, the lower bounds $r_g=u_g\ind(d^M(A_g)\geq u_g)$ satisfy (\ref{gen}), and a procedure which rejects $H_0^{u_g/n_g}$ if $r_g=u_g$ controls the FWER on any subset of partial conjunction hypotheses. In addition, for any subset of groups with index set $K\subseteq\{1, \ldots, G\},$ one may obtain a lower bound for the number of false partial conjunction hypotheses in the set $\{H_0^{u_g/n_g}, \,\, g\in K\},$ by computing $d^G(K)=\sum_{g\in K} \ind(d^M(A_g)\geq u_g),$ and these are simultaneous $(1-\alpha)$-confidence lower bounds for all subsets $K\in\{1, \ldots, G\}$ provided that $d^M(S)$ satisfy (\ref{lower}).  Therefore, flexible inference on a-priori given groups of hypotheses $\{A_g, \,\,g=1, \ldots, G\}$ can be obtained by an FDP-controlling procedure, outputting $d^M(S)$ satisfying (\ref{lower}). However, if only FWER or FDP control on the family of $\{H_0^{u_g/n_g}, \,\,g=1, \ldots, G\}$ is of interest,  which can be obtained via control of (\ref{gen}) with $r_g\in\{0, u_g\},$
then one may apply an FDP-controlling procedure on the set of partial conjunction $p$-values $\{P_g^{u_g/n_g}, \,\,g=1, \ldots, G\}$ rather than on all the elementary $p$-values $\{p_1,\ldots,p_M\},$ which may offer a power gain. The results of this paper may be useful for obtaining validity of such methods under different dependency structures within and across the groups. We discuss this next.

Goeman and Solari (\cite{goeman2011multiple}) showed that closed testing can be used for obtaining FDP control (\ref{lower}). Several other procedures for controlling the FDP have been developed (see, e.g., \cite{genovese2006exceedance}, \cite{blanchard2020post}, and \cite{katsevich2020simultaneous}). However, Goeman et al. \cite{goeman2021only} showed that these procedures, as well as other procedures targeting control of tail probabilities of the FDP, such as FWER-controlling procedures, are either equivalent to a closed testing procedure, or can be uniformly improved by one. The properties of the closed testing procedure rely on the properties of the local tests used for testing the intersection hypotheses: the procedure is valid for FDP control if all its local tests are valid, and the more powerful the local tests are, the more powerful is the FDP-controlling procedure. As noted in Section \ref{MT}, if the family of hypotheses of interest is $\{H_0^{u_g/n_g}, \,\,g=1, \ldots, G\},$ then closed testing may receive as input the partial conjunction $p$-values for these hypotheses, and test the intersections of these hypotheses. This approach can be used for control of FWER or FDP on the partial conjunction hypotheses, or for obtaining lower bounds $r_g\in \{0,u_g\}$ satisfying (\ref{gen}), as long as the local tests for testing intersections of partial conjunction hypotheses are valid. While full closed testing applied on all the elementary hypotheses $H_1, \ldots, H
_M$ can also be used for these purposes (because $r_g\in \{0,u_g\}$ satisfying (\ref{gen}) may be obtained from $d^M(S)$ satisfying (\ref{lower}), as discussed above), the approach of applying closed testing directly on the partial conjunction hypotheses may be more powerful, because of restricting attention only to a certain subset of intersections of the elementary hypotheses. Corollary \ref{cor-sim} shows three cases in which Simes' test is valid for testing an intersection of partial conjunction hypotheses, yielding validity of closed testing procedure on partial conjunction $p$-values with Simes' local tests. For example, under arbitrary dependence within each group and positive dependence across the groups in the sense of item D4 in Section \ref{setting}, closed testing procedure with Simes' local tests applied on Bonferroni partial conjunction $p$-values is valid for FDP control.  Considering alternative valid approaches which account for arbitrary dependence, such as applying Bonferroni local tests on partial conjunction $p$-values, or applying full closed testing with Bonferroni local tests, would result in less powerful procedures. The results of Corollary \ref{cor-sim} could be useful  even if the goal is FDP control on all subsets of elementary hypotheses, given by (\ref{lower}): one could gain power by replacing certain local tests by more powerful valid tests, possibly applied on global null group $p$-values rather than on elementary $p$-values. 

Similarly to the power of the FDP-controlling procedures, the power of the FDR-controlling procedures  applied on partial conjunction $p$-values strongly depends on the choice of partial conjunction tests. In this work we developed sufficient conditions for FDR control on partial conjunction hypotheses for the cases where their $p$-values are connected to self-consistent multiple testing procedures, or are constructed using Fisher's or Stouffer's methods, showing that in several cases arbitrary dependence adjustments are not needed for FDR control. It may be of interest to address other partial conjunction tests, such as those based on combining methods of Pearson (\cite{pearson1934new}, \cite{owen2009karl}), Tippett (\cite{tippett1931methods}),  Vovk and Wang (\cite{vovk2020combining}), Wilson (\cite{wilson2019harmonic}),  the Higher Criticism method (\cite{DJ04}), or methods based on $e$-values (see \cite{grunwald2020safe}, \cite{vovk2019values}). Obviously, when the partial conjunction $p$-values are combinations of within-group $p$-values, and the $p$-values belonging to different groups are independent or are arbitrarily dependent, the results of \cite{blanchard2008two} can be used for obtaining sufficient conditions for FDR control on partial conjunction hypotheses, 
as long as the dependency structure within the groups is such that the combined $p$-values are valid for partial conjunction testing. The more interesting cases are when there are dependencies across the groups, or dependencies within the studies in the meta-analysis setting, possibly implying complex dependencies among the partial conjunction $p$-values. The arbitrary dependence adjustments may lead to a substantial power loss, therefore searching for realistic dependency structures which do not require these adjustments for certain powerful partial conjunction tests may be an interesting direction for future research. 

One of our results addresses Simes-Storey $p$-value (\ref{PC-Storey}) for the partial conjunction hypothesis, which is based on the minimum adjusted $p$-value of the adaptive BH procedure with Storey's plug-in estimator. In particular, we showed that in the meta-analysis setting, when there is positive dependence within each study in the sense of item M1 in Section \ref{meta-mult}, the (doubly-weighted) BH procedure controls the FDR on the family of partial conjunction hypotheses when each such hypothesis is tested using Simes-Storey $p$-value. In addition, we showed that in this case the procedure for assessing replicability of findings in Section \ref{proc-rep} controls the expected fraction of false replicability claims for the selected features. Similar results were shown for Stouffer's and Fisher's partial conjunction tests. It may be of interest to study the power of Simes-Storey partial conjunction test in comparison to those tests, as well as possibly other tests which are valid under independence.  While there are several papers which compared the power of different methods for global null testing under independence (see, e.g., \cite{loughin2004systematic}, \cite{owen2009karl}), and in a recent work  by \cite{hoang2021combining} several partial conjunction tests were compared under different scenarios, it seems that Simes-Storey partial conjunction test has not been studied and compared to other tests. 
It may be of interest to find the scenarios in which this test is more powerful than others. In addition, following the empirical results of \cite{blan2009}, showing that the adaptive BH procedure incorporating Storey's plug-in estimator with $\lambda=\alpha$  controls the FDR under positive dependence, one may expect that  the partial conjunction test, based on Simes-Storey $p$-value with this choice of $\lambda,$ is valid under positive dependence. Obtaining this result may give the possibility to replace Simes' method by Simes-Storey method for combining positively dependent $p$-values. The latter method is expected to be more powerful than the former for the groups which contain a large proportion of false null hypotheses, so investigation of the validity of Simes-Storey partial conjunction test under positive dependence is also an interesting direction for future research. It may  be also worth to investigate the performance of partial conjunction tests based on  adaptive FDR-controlling procedures incorporating other plug-in estimators satisfying Condition \ref{cond-adapt}, such as those considered in Corollary 13 of \cite{blan2009}. 

When the hypotheses have a group structure, testing the global null 
hypotheses for the groups may be performed for selecting promising groups of hypotheses, which may be only a first step, 
after which the elementary hypotheses within the selected groups are tested. 
The global null hypotheses may be replaced by partial conjunction hypotheses with $u\geq2,$ if one is in search for groups with at least $u$ signals. This two-step approach, with global null testing in the first step, was addressed by \cite{BB14}, and  can be viewed as hierarchical testing of a tree, which was introduced in the context of FDR control by  \cite{YetN06}, and was further developed by \cite{Y08}. These papers considered a general tree of hypotheses, where each hypothesis is associated with a single parent hypothesis, except for the hypotheses at the first level of the tree. In our setting, each group of elementary hypotheses may be associated with a parent, its partial conjunction hypothesis, which gives us a two-level tree structure, with partial conjunction hypotheses at the first level and elementary hypotheses at the second level. Yekutieli et al. (\cite{YetN06}) considered several types of FDR which may be relevant when addressing a tree of hypotheses, including level-$l$ FDR, which addresses only discoveries at a specific level $l,$ chosen in advance. For the above two-level structure, the results of this work are useful for obtaining sufficient conditions for level-1 FDR control. While \cite{Y08} considered independence of the $p$-values across the entire tree, we consider trees with dependencies within and across the levels. The dependence between the level-1 and level-2 $p$-values follows from the fact that the partial conjunction $p$-values are combinations of within-group elementary $p$-values, and the dependence between level-1 $p$-values is induced by the dependencies within the level-2 $p$-values. Several papers (\cite{RetJ17}, \cite{ramdas2019sequential}, \cite{bogomolov2020hypotheses}, \cite{filtering}) considered general trees where each parent hypothesis is the intersection of its child hypotheses, and their $p$-values are computed either by combining the $p$-values of their child hypotheses, or by combining the $p$-values of all their descendants residing at the finest level of the tree (which are referred to as leaf hypotheses). The lemmas developed in this paper generalize the results of the group-level superuniformity lemma of \cite{RetJ17}, and may be used for extending the theoretical guarantees of the methods developed in the above papers, 
as well as possibly other methods for FDR control on a tree of hypotheses (see, e.g., \cite{LG16} and \cite{Y08}).
Finally, an interesting research direction is generalizing  the results of this work for obtaining sufficient conditions for overall and level-restricted FDR control on a tree of hypotheses with an arbitrary number of levels, addressing various combining methods for obtaining $p$-values for parent  hypotheses, in conjunction with different dependency structures among the elementary $p$-values.

\section*{Acknowledgements}
I am thankful to Yoav Benjamini, Jelle Goeman, Ruth Heller, Eugene Katsevich, Abba Krieger, Oren Louidor, Aaditya Ramdas, Chiara Sabatti, and Aldo Solari for useful discussions. This research was partially supported by the Israel Science Foundation grant no. 1726/19.
\appendix
\renewcommand{\thethm}{A\arabic{thm}}
\setcounter{thm}{0}
\section{Sufficient conditions for FDR control}\label{appendix:suff}
Several results in this paper are proved by showing that the two sufficient conditions for FDR control, given by \cite{blanchard2008two}, are satisfied in the context of multiple testing of partial conjunction hypotheses. For completeness, these sufficient results are given in item 1 of the proposition below. An additional set of similar conditions is given in item 2 of this proposition, and is also useful for obtaining some results in this paper. 
This section addresses the setting and the definitions  of Section \ref{prel}.
\begin{proposition}\label{SM-suffic}
 Let  $\beta: \mathbb{R}^+\rightarrow \mathbb R^{+}$ be a non-decreasing shape function, and $\alpha$ a positive number. Consider a multiple testing procedure $\mathcal{M}$ with thresholds of form (\ref{non-adapt}), with shape function $\beta$ and target FDR level $\alpha,$ which receives as input the $p$-values $\bm{p}$ and outputs the set of rejected hypotheses $\mathcal{R}(\bm{p}).$  In each of the following two cases, the multiple testing procedure $\mathcal{M}$ guarantees $$\text{FDR}_v\equiv\EE{\frac{\sum_{i\in M_0}v_i\ind(i\in \mathcal{R})}{\sum_{i=1}^mv_i\ind(i\in \mathcal{R})}} \leq \frac{\alpha}{m}\sum_{i\in M_0}w_iv_i\leq \alpha.$$
 \begin{enumerate}
     \item The multiple testing procedure $\mathcal{M}$ is self-consistent with respect to its thresholds, and for each $i\in M_0,$ 
 		the couple $(p_i, |\mathcal{R}|_{\bm v})$ satisfies the dependency control condition with respect to the shape function $\beta.$
 		\item The multiple testing procedure $\mathcal{M}$ is stable and self-consistent with respect to its thresholds, and for each $i\in M_0,$ the couple  $(p_i, |\mathcal{R}^{-i}|_{\bm v})$ satisfies the dependency control condition with respect to the shape function $\beta.$ 
 \end{enumerate}
\end{proposition}
Item 1 was proved by \cite{blanchard2008two}, see their Proposition 2.7. The proof of item 2 is similar.
\begin{proof}[Proofs of item 2 of Proposition \ref{SM-suffic}]
Assume the multiple testing procedure is stable and self-consistent with respect to thresholds of form (\ref{non-adapt}) with a certain shape function $\beta,$ and the couple $(p_i, |\mathcal{R}^{-i}|_{\bm v})$ satisfies the dependency control condition with respect to the same shape function $\beta.$ Then, for this   multiple testing procedure,
\begin{align}
\text{FDR}_v&=\EE{\frac{\sum_{i \in M_0}v_i\ind(i\in \mathcal{R})}{\sum_{i=1}^mv_i\ind(i\in \mathcal{R})}}=\sum_{i \in M_0}v_i\EE{\frac{\ind(i\in \mathcal{R})}{|\mathcal{R}|_{\bm v}}}\notag\\&=
\sum_{i \in M_0}v_i\EE{\frac{\ind\{i\in \mathcal{R}, p_i\leq w_i\beta(|\mathcal{R}|_{\bm v})\alpha/m\}}{|\mathcal{R}|_{\bm v}}}\label{cons}\\&=
\sum_{i \in M_0}v_i\EE{\frac{\ind\{i\in \mathcal{R}, p_i\leq w_i\beta(|\mathcal{R}^{-i}|_{\bm v})\alpha/m\}}{|\mathcal{R}^{-i}|_{\bm v}}}\label{stable}\\&\leq \sum_{i \in M_0}v_i\EE{\frac{\ind\{p_i\leq w_i\beta(|\mathcal{R}^{-i}|_{\bm v})\alpha/m\}}{|\mathcal{R}^{-i}|_{\bm v}}}.\label{last-suffic}
\end{align} 
where the equality in (\ref{cons}) follows from the fact that the multiple testing procedure is self-consistent with respect to thresholds of form (\ref{non-adapt}), i.e. $\mathcal{R}\subseteq \{i: p_i\leq w_i\beta(|\mathcal{R}|_v)\alpha/m\}.$ The equality in (\ref{stable}) follows from the fact that the procedure is stable, yielding that if $i\in \mathcal{R},$ then $\mathcal{R}=\mathcal{R}^{-i}.$  Now, using the fact that for each $i\in M_0,$ the pair $(p_i, |\mathcal{R}^{-i}|_{\bm v})$ satisfies the dependency control condition with respect to shape function $\beta,$ we obtain for each $i\in M_0,$
$$\EE{\frac{\ind\{p_i\leq w_i\beta(|\mathcal{R}^{-i}|_{\bm v})\alpha/m\}}{|\mathcal{R}^{-i}|_{\bm v}}}\leq \frac{w_i\alpha}{m}.$$  Combining this inequality with (\ref{last-suffic}), we obtain
$$\text{FDR}_v\leq \frac{\alpha}{m} \sum_{i \in M_0}v_iw_i.$$
\end{proof}
\section{Proofs}
\subsection{Proof of Proposition \ref{main-FDR}}
\textbf{Proof of item 1.1.} It is assumed that for each group $g=1, \ldots, G,$ the partial conjunction $p$-value $P_g^{u_g/n_g}$ is a combination of $p$-values within group $g.$ In addition, it is assumed that the p-values in each group are independent of the p-values in any other
group. Combining these two facts, we obtain that the partial conjunction $p$-values $P^{u_g/n_g}, g=1, \ldots, G$ are independent. Now, since  the multiple testing procedure is assumed to be non-increasing, Proposition 3.3 of \cite{blanchard2008two} yields that  for each group $g$ with a true null hypothesis $H_0^{u_g/n_g},$ the pair $(P^{u_g/n_g}, |\mathcal{R}|_{\bm v})$ satisfies the dependency control with identity shape function $\beta(x)=x$. Now, since the multiple testing procedure is self-consistent with respect to thresholds of form (\ref{non-adapt}), we obtain that the two sufficient conditions in Proposition 2.7 of \cite{blanchard2008two} are satisfied, and the result follows.
\\\\
\textbf{Proof of item 1.2.} As noted in the proof of item 1.1, since there is independence across the groups, the partial conjunction $p$-values are independent. Therefore, all the assumptions of Theorem 11 of \cite{blan2009} hold, and the result follows.
\\\\\textbf{Proofs of items 2 and 3.} The result of item 2 follows from Propositions 3.6 and 2.7 of \cite{blanchard2008two}. The result of item 3 follows from Propositions 3.7 and 2.7 of \cite{blanchard2008two}.

\subsection{Proof of Theorem \ref{min-adj-non-adapt}}
\textbf{Proof of item 1.} Let $|\mathcal{R}(\bm{p})|_{\bm v}$ be the  weighted number of rejected hypotheses when the given multiple testing procedure $\mathcal{M}$ is applied on $P_g^{u_g/n_g}, g=1, \ldots, G.$ Since it is assumed that for each group $g,$ the procedure $\mathcal{M}_g$ satisfies the monotonicity property given in Remark \ref{monotone}, $P_g^{u_g/n_g}$ is non-decreasing in each coordinate of $\bm{p}$. In addition, it is assumed that the procedure $\mathcal{M}$ is non-increasing. Using these facts we obtain that $|\mathcal{R}(\bm{p})|_{\bm v}$ is non-increasing in each coordinate of $\bm{p}.$ Therefore, using item 1 of Lemma \ref{ov-PRDS1} we obtain that for each group $g$ such that $H_0^{u_g/n_g}$ is a true null hypothesis, the pair $(P_g^{u_g/n_g}, |\mathcal{R}(\bm{p})|_{\bm v})$ satisfies the dependency control condition with the identity shape function $\beta(x)=x.$ Since it is assumed that $\mathcal{M}$ is self-consistent with respect to thresholds of form (\ref{non-adapt}) with $\beta(x)=x,$, we obtain that the two sufficient conditions of Proposition 2.7 of \cite{blanchard2008two}, given in item 1 of Proposition \ref{SM-suffic}, are satisfied, and the result follows.
\\\\\textbf{Proof of item 2.}
 Consider the rejection set $\mathcal{R}\subseteq\{1, \ldots, G\}$  of the multiple testing procedure $\mathcal{M}$ when it receives as input the partial conjunction $p$-values $P_g^{u_g/n_g}, \,g=1, \ldots, G.$ Since $\mathcal{M}$ is stable, one may define for each $g\in \{1, \ldots, G\}$ the set $\mathcal{R}^{-g}$ 
 as the rejection set which is obtained when all the partial conjunction $p$-values except $P_g^{u_g/n_g}$ are fixed, and $P_g^{u_g/n_g}$ has any value such that $H_0^{u_g/n_g}$ is rejected. Since each partial conjunction $p$-value $P_g^{u_g/n_g}$ is a combination of $p$-values in group $g,$ and $\mathcal{R}^{-g}$ is a function of all the partial conjunction $p$-values except $P_g^{u_g/n_g},$ we obtain that $\mathcal{R}^{-g}$ is a function of all the elementary $p$-values which do not belong to group $g,$ i.e. $\mathcal{R}^{-g}=\mathcal{R}^{-g}(\bm p^{-g}).$ 
 Since $\mathcal{M}$ is concordant, 
 for each $g\in \{1, \ldots, G\},$ $|\mathcal{R}^{-g}|_{\bm v}$ is non-increasing in each partial conjunction $p$-value in the set $\{P_k^{u_k/n_k}, \,k\neq g\}.$  Each partial conjunction $p$-value is non-decreasing in each coordinate of $\bm{p},$ as shown in the proof of item 1 of Theorem \ref{min-adj-non-adapt}. Therefore, for each $g\in\{1, \ldots, G\},$
 $|\mathcal{R}^{-g}|_{\bm v}$ is non-increasing in each coordinate of $\bm{p}^{-g}.$ Hence, according to item 2 of Lemma \ref{ov-PRDS1}, for each group $g\in \mathcal{G}_0,$ the pair $(P_g^{u_g/n_g}, |\mathcal{R}^{-g}|_{\bm v})$ satisfies the dependency control condition with the identity shape function $\beta(x)=x.$ In addition, the multiple testing procedure $\mathcal{M}$ is self-consistent with respect to thresholds of form (\ref{non-adapt}) with $\beta(x)=x.$ Therefore, the conditions of item 2  of Proposition \ref{SM-suffic} are satisfied, and the result follows.
  \\\\\textbf{Proof of item 3.}
 The proof is similar to the proof of item 2 above. As shown in that proof, for each $g\in\{1, \ldots, G\},$ $|\mathcal{R}^{-g}|_{\bm v}$ is non-increasing in each coordinate of $\bm{p}^{-g}.$ Hence, according to item 3 of Lemma \ref{ov-PRDS1}, for each group $g\in \mathcal{G}_0,$ the pair $(P_g^{u_g/n_g}, |\mathcal{R}^{-g}|_{\bm v})$ satisfies the dependency control condition with the identity shape function $\beta(x)=x.$ Now the result follows from item 2 of Proposition \ref{SM-suffic}.
 
 \subsection{Proof of Corollary \ref{cor}}\label{proofs:corBH}
 The corollary follows from Theorem \ref{min-adj-non-adapt}. Let us first show that any step-up procedure with thresholds of form (\ref{non-adapt}), in particular the doubly-weighted BH procedure, is self-consistent, non-increasing, stable, and concordant.  Blanchard and Roquain (\cite{blanchard2008two}) showed that a step-up procedure with thresholds $\Delta(i,r)$ of form (\ref{non-adapt}) rejects the hypotheses in the set $L_{\Delta}(\hat{r}),$ where $\hat{r}=\max\{r: |L_{\Delta}(i, r)|_{\bm v}\geq r\},$ and $|L_{\Delta}(\hat{r})|_{\bm v}=\hat{r},$ therefore, it  satisfies the self-consistency condition (\ref{SC}) with equality.  Obviously, for each $r\geq 0,$ $|L_{\Delta}(i, r)|_{\bm v}$ is non-increasing in each $p$-value, therefore $\hat{r}=|L_{\Delta}(\hat{r})|_{\bm v}$ is non-increasing in each $p$-value. 
 It remains to show stability and concordance. 
 Let us fix $i\in\{1, \ldots, m\}.$ For each $r\geq 0,$ let
$$ L^{(i)}_{\Delta}(r)=\{j\neq i: p_{j}\leq \Delta(i, r)\}.$$
 Define 
 $$\hat{r}^{(i)}=\max\{r: |L^{(i)}_{\Delta}(r)|_{\bm v}\geq r-v_i\}.$$
 Note that $\hat{r}^{(i)}$ is a function of $\bm p^{-i}.$ Let us show that if $i\in \mathcal{R},$ then $\hat{r}=\hat{r}^{(i)}.$
 Assume that $i\in \mathcal{R}.$ Then $p_i\leq \Delta(i, \hat{r}),$
 where $\hat{r}=\max\{r: |L_{\Delta}(r)|_{\bm v}\geq r\}.$ For any $r\geq \hat{r},$ $p_i\leq \Delta(i, r),$
which yields that $|L_{\Delta}(r)|_{\bm v}=v_i+ |L^{(i)}_{\Delta}(r)|_{\bm v}.$ Therefore, according to the definition of $\hat{r},$ $|L^{(i)}_{\Delta}(\hat{r})|_{\bm v}\geq \hat{r}-v_i,$ and for any $r> \hat{r},$ $|L^{(i)}_{\Delta}(r)|_{\bm v}<r-v_i.$ Thus we obtain that $\hat{r}=\hat{r}^{(i)}=\max\{r: |L^{(i)}_{\Delta}(r)|_{\bm v}\geq r-v_i\}.$ Therefore, fixing $\bm p^{-i}$ and changing $p_i$ as long as $i\in \mathcal{R}$ will not change the set $\mathcal{R},$ which will be equal to $\mathcal{R}^{-i}(\bm p^{-i})=\{i\}\cup L_{\Delta}^{(i)}(\hat{r}^{(i)}).$ Note that for each $r\geq 0,$ $|L^{(i)}_{\Delta}(r)|_{\bm v}$ is non-increasing in each coordinate of $\bm p^{-i}$, therefore $\hat{r}^{(i)}$ is non-increasing in each coordinate of $\bm p^{-i}.$ Since $|L^{(i)}_{\Delta}(r)|_{\bm v}$ is non-decreasing in $r,$ we obtain that $|L_{\Delta}^{(i)}(\hat{r}^{(i)})|_{\bm v}$ is non-increasing in each coordinate of $\bm p^{-i}.$ Therefore, this is also true for $|\mathcal{R}^{-i}(\bm p^{-i})|_{\bm v}=v_i+|L_{\Delta}^{(i)}(\hat{r}^{(i)})|_{\bm v}.$ Thus we have shown that any step-up procedure with thresholds of form (\ref{non-adapt}) is stable and concordant.

Now, it is enough to show that the partial conjunction $p$-values given in each item of Corollary \ref{cor} satisfy the conditions of the corresponding item of Theorem \ref{min-adj-non-adapt}. For item 1, let us consider the partial conjunction $p$-values based on Simes', Hommel's and Bonferroni's methods.   The partial conjunction $p$-value based on Simes' method is connected to the BH procedure in the sense of (\ref{PCM}), because Simes' global null $p$-value is the minimum adjusted $p$-value of the BH procedure. The BH procedure is non-increasing and is self-consistent with respect to thresholds of form (\ref{non-adapt}) with prior and penalty weights equal to unity and identity shape function, therefore Simes' partial conjunction $p$-value satisfies condition (c) of item 1 of Theorem \ref{min-adj-non-adapt}. The global null $p$-values based on Hommel's and Bonferroni's methods (given in (\ref{Hommel}) and (\ref{Bonf}), respectively), are also connected in the sense of (\ref{PCM}) to non-increasing procedures (the BY and Bonferroni procedures, respectively).  Since these procedures always reject a subset of the rejection set of the BH procedure, they satisfy the self-consistency property given in (\ref{SC}) with respect to thresholds of form (\ref{non-adapt}) with prior and penalty weights equal to unity and identity shape function. Therefore,
the global null $p$-values based on Hommel's and Bonferroni's methods satisfy condition (c) of item 1 of Theorem \ref{min-adj-non-adapt}. 
 This completes the proof of item 1 of Corollary \ref{cor}. For proving item 2 of Corollary \ref{cor}, we have to consider the partial conjunction $p$-value (\ref{PC-Storey}), which is connected in the sense of (\ref{PCM}) to the adaptive BH procedure with Storey's estimator for the proportion of nulls. This procedure is self-consistent with respect to thresholds of form (\ref{adapt1}), and Storey's estimator for the proportion of nulls satisfies Condition \ref{cond-adapt} (see Section \ref{prel}). This completes the proof of item 2. Item 3 assumes that the partial conjunction $p$-values are based on Bonferroni's method, as well item 3 of Theorem \ref{min-adj-non-adapt}. Therefore, the proof of Corollary \ref{cor} is complete. 
 \subsection{Proof of Theorem \ref{thm-meta}}
 
 The proof of item 1 is based on the following lemma.
 \begin{lem}\label{lemma-prob}
 Assume the meta-analysis setting of Section \ref{setting-metaanalysis}, where the hypotheses are arranged in a matrix with $m$ rows and $n$ columns. Assume that for each $j\in\{1, \ldots, n\},$  $\bm p_{\cdot j}$ is independent of $\bm p_{\cdot(-j)},$ and  $\bm p_{\cdot j}$ satisfies the PRDS property on the subset of true null hypotheses.  Then the following results hold:
 \begin{enumerate}
     \item The vector of all the $n\times m$ $p$-values $\bm p$ satisfies the PRDS property on the subset of true null hypotheses.
     \item When the rows of the matrix define groups of hypotheses, the elementary $p$-values satisfy positive dependence across groups, as defined in item D5 of Section \ref{setting}.
     \item When the rows of the matrix define groups of hypotheses, the elementary $p$-values satisfy conditional positive dependence across groups and independence within each group, as defined in items D4 and D2 of Section \ref{setting}, respectively.
 \end{enumerate}
 \end{lem}
 The proof of this lemma is given in the end of this section. The results of this lemma show that under independence across the studies and positive dependence within each study, as defined in item M1 in Section \ref{meta-mult}, the $p$-values satisfy the dependency structure assumed in items 1, 2, and 3 of Theorem \ref{min-adj-non-adapt}. The other assumptions of items 1 and 2 of Theorem \ref{min-adj-non-adapt} are assumed to be satisfied, therefore the result of item 1 of Theorem \ref{thm-meta} follows from items 1 and 2 of Theorem \ref{min-adj-non-adapt}, when we address the hypotheses within each row of the matrix as a group.
 
 Let us now prove item 2 of Theorem \ref{thm-meta}.
According to Lemma \ref{lemma-meta}, for each $i\in\{1, \ldots,m\}$ such that $H_i^{u/n}$ is true, the pair $(P_i^{u/n}, f(\bm p))$ satisfies the dependency control condition with respect to the identity shape function $\beta(x)=x,$ for any non-increasing function $f: [0,1]^M\rightarrow [0, \infty).$ Since the multiple testing procedure $\mathcal{M}$ is assumed to be non-increasing, and both Stouffer's and Fisher's partial conjunction $p$-values are non-decreasing in each elementary $p$-value, we obtain that  for each $i\in\{1, \ldots,m\}$ such that $H_i^{u/n}$ is true, the pair $(P_i^{u/n}, |\mathcal{R}(\bm p)|_{\bm v})$ satisfies the dependency control condition with respect to the identity shape function $\beta(x)=x.$ Since $\mathcal{M}$ is assumed to be self-consistent with respect to thresholds of form (\ref{non-adapt}) with identity shape function $\beta(x)=x,$
we obtain that the two sufficient conditions for FDR control, given in Proposition 2.7 of \cite{blanchard2008two} (and item 1 of Proposition \ref{SM-suffic}), are satisfied, which completes the proof.

 \begin{proof}[Proof of Lemma \ref{lemma-prob}]
Items 1 and 2 of Lemma \ref{lemma-prob} are based on the following auxiliary lemma.
\begin{lem}\label{prob}
Let $\bm p_1\in[0,1]^{n_1}$ and $\bm p_2\in[0,1]^{n_2}$ be two independent vectors, and assume that $\bm p_1$ satisfies the PRDS property on the subset $I_0\subseteq \{1, \ldots, n_1\}.$ Then  the vector obtained by stacking $\bm p_1$ and $\bm p_2,$ denoted by  $(\bm p_1, \bm p_2),$ satisfies the PRDS property on the subset $I_0.$ 
\end{lem}
Let us first show how item 1 of Lemma \ref{lemma-prob} follows from Lemma \ref{prob}. Let $H_{ij}$ be a true null hypothesis, and let $\bm p$ be the vector of all the $m\times n$ $p$-values. 
Note that $\bm p$ can be partitioned into two vectors: $\bm p_{\cdot j},$ the vector of $p$-values in study $j,$ and $\bm p_{\cdot (-j)},$ the vector of $p$-values for the hypotheses in all the studies except study $j.$ 
Since $\bm p_{\cdot j}$ is PRDS on the subset of true null hypotheses in study $j,$ and $\bm p_{\cdot j}$ is independent of $\bm p_{\cdot (-j)},$ (because of independence across studies), we obtain based on Lemma \ref{prob} that $\bm p$ satisfies the PRDS property on the subset of indices corresponding to true null hypotheses in study $j.$ Therefore, for any non-decreasing set $D$ in $m\times n,$ $\PP{\bm p\in D\,|\,p_{ij}=x}$ is non-decreasing in $x.$
Thus we have proved item 1 of Lemma \ref{lemma-prob}. Item 2 of Lemma \ref{lemma-prob} follows similarly. Since the groups are defined by rows, independence within each group follows from the independence across studies: for each $i\in\{1, \ldots, m\},$ the vector $\bm p_{i\cdot}$ contains independent $p$-values.  Let us now show positive dependence across groups. Let $i\in\{1, \ldots, m\}$ be a row (group) which contains at least one true null hypothesis, and let $H_{ij}$ be a true  null hypothesis within this row. Let $D$ be a non-decreasing set in $[0,1]^{mn-m+1}.$ We have to show that $\mathbb{P}\left\{(\bm{p}_{(-i)\cdot}, p_{ij})\in D\,|\,p_{ij}=x \right\}$ is non-decreasing in $x.$ Note that $(\bm{p}_{(-i)\cdot}, p_{ij})=(\bm{p}_{(-i)(-j)}, \bm p_{\cdot j}),$ where $\bm{p}_{(-i)(-j)}$ is the vector obtained by stacking the rows of matrix $\mathbf{p}$ after removing its row $i$ and column $j.$ According to our assumptions, $\bm{p}_{(-i)(-j)}$ is independent of $\bm p_{\cdot j},$ while  $\bm p_{\cdot j}$ is PRDS on the subset of true null hypotheses in study $j.$ Therefore, according to Lemma \ref{prob}, $(\bm{p}_{(-i)(-j)}, \bm p_{\cdot j})=(\bm{p}_{(-i)\cdot}, p_{ij})$ is PRDS on the subset of true null hypotheses in study $j,$ which yields that 
$\mathbb{P}\left\{(\bm{p}_{(-i)\cdot}, p_{ij})\in D\,|\,p_{ij}=x \right\}$ is non-decreasing in $x.$ This completes the proof of item 2 of Lemma \ref{lemma-prob}.

Let us now prove item 3 of Lemma \ref{lemma-prob}. 
Let $\bm y\in[0,1]^{n-1},$ and let $D\in[0,1]^{mn}$ be a non-decreasing set. Let $H_{ij}$ be a true null hypothesis, and let $\bm p_{i(-j)}$ be the vector of $p$-values for the hypotheses in $\{H_{ik}, k\neq j\}.$ Finally, let 
$(\bm p_{(-i)\cdot}, p_{ij}, \bm y)$ be the vector $\bm p$ such that the sub-vector corresponding to the hypotheses not belonging to row $i$ is given by $\bm p_{(-i)\cdot},$ the $p$-value for $H_{ij}$ is $p_{ij},$ and the sub-vector of $p$-values for the remaining hypotheses, belonging to the set $\{H_{ik}, k\neq j\},$ is $\bm y.$
Note that
\begin{align}&\PP{\bm p\in D\,|\,p_{ij}=x, \bm p_{i(-j)}=\bm y }=\frac{\PP{(\bm  p_{(-i)\cdot}, p_{ij})\in \tilde{D}\,|\,p_{ij}=x}}{f_{\bm p_{i(-j)}|p_{ij}=x}(\bm y)},\label{med2}
\end{align}
where 
$$\tilde{D}=\left\{(\bm p_{i(-j)}, p_{ij}): (\bm p_{i(-j)}, p_{ij}, \bm y)\in D\right\},$$ 
and $f_{\bm p_{i(-j)}|p_{ij}=x}$ is the conditional density of $\bm p_{i(-j)}$  given $p_{ij}=x.$ Since the $p$-values for the hypotheses in row $i$ are independent, this conditional density is the density of $\bm p_{i(-j)}.$ Therefore, the denominator in (\ref{med2}) does not depend on $x.$
Since $D$ is a non-decreasing set, $\tilde{D}$ is also a non-decreasing set. According to item 2, the $p$-values satisfy positive dependence across groups, therefore the nominator of (\ref{med2}) is non-decreasing in $x.$ We obtain that the fraction in (\ref{med2}) is non-decreasing in $x.$ Thus we have proved conditional positive dependence across groups. The independence within each group, i.e. row, follows from the independence across studies.

Let us now prove Lemma \ref{prob}.	Let $D$ be a non-decreasing set in $[0,1]^{n_1+n_2}.$ Let $i\in I_0.$
Since $\bm p_2$ is independent of $\bm p_{1i},$ the conditional distribution of $\bm p_2$ given $p_{1i}$ is identical to the distribution of $\bm p_2.$ Therefore, for any $x\in[0,1],$
\begin{align}
\EE{\ind(\bm p_1, \bm p_2)\in D\,|\,p_{1i}=x}&=\EE{\EE{\ind(\bm p_1, \bm p_2)\in D\,|\,p_{1i}=x, \bm p_2}}\label{prob1}
\end{align}
where the external expectation on the right hand-side of (\ref{prob1})
is with respect to the distribution of $\bm p_2.$  Let $\bm y\in[0,1]^{n_2}.$	Since $\bm p_1$ and $\bm p_2$ are independent, for any $x\in[0,1],$
\begin{align}
\EE{\ind(\bm p_1, \bm p_2)\in D\,|\,p_{1i}=x, \bm p_2=\bm y}=\EE{\ind(\bm p_1, \bm y)\in D\,|\,p_{1i}=x}.\label{prob2}
\end{align}
The set $\{\bm p_1: (\bm p_1, \bm y)\in D\}$ is non-decreasing, because the set $D$ is non-decreasing. Therefore, since $\bm p_1$ is PRDS with respect to $p_{1i},$ $$\EE{\ind(\bm p_1, \bm y)\in D\,|\,p_{1i}=x}$$ is a non-decreasing function of $x.$ Based on (\ref{prob1}) and (\ref{prob2}), we obtain
\begin{align}\EE{\ind(\bm p_1, \bm p_2)\in D|p_{1i}=x}=\int_{\bm p\in[0,1]^{n_2}}\EE{\ind(\bm p_1, \bm p)\in D|p_{1i}=x}f_{\bm p_2}(\bm p)d\bm p,\label{integral}\end{align}
where $f_{\bm p_2}$ is the density function of $\bm p_2.$ As shown above, for each $\bm p\in[0,1]^{n_2},$ $\EE{\ind(\bm p_1, \bm p)\in D\,|\,p_{1i}=x}$ is a non-decreasing function of $x.$ Therefore, based on (\ref{integral}) we obtain that 
$\EE{\ind(\bm p_1, \bm p_2)\in D\,|\,p_{1i}=x}$ is a non-decreasing function of $x,$ which completes the proof.
\end{proof}
\subsection{Proof of Theorem \ref{meta-seq}}
 The proof is based on the following lemma.
\begin{lem}\label{lemma-seq}
The procedure in Section \ref{proc-rep} with shape function $\beta(\cdot)$  guarantees \begin{align}\EE{\frac{\sum_{i\in \mathcal{S}}v_i\ind(\hat{k}(i)>k(i))}{|\mathcal{S}|_{\bm v}}}\leq q\label{result}
\end{align}if either of the following conditions is satisfied:
\begin{enumerate}
    \item 
    For each $i\in\{1, \ldots, n\}$ such that $k(i)<n,$ the pair $(P_i^{(k(i)+1)/n}, |\mathcal{S}|_{\bm v})$ satisfies the dependency control condition with shape function $\beta.$
    \item The selection rule $\mathcal{S}$ is stable (as defined in Condition C2), and for each $i\in\{1, \ldots, n\}$ such that $k(i)<n,$ the pair $(P_i^{(k(i)+1)/n}, |\mathcal{S}^{(-i)\cdot}|_{\bm v})$ satisfies the dependency control condition with shape function $\beta.$
\end{enumerate}
\end{lem}
The proof of this lemma is given in the end of this section. Let us first prove item 1(a) of  Theorem \ref{meta-seq}. For each $i\in\{1, \ldots, m\}$ such that $k(i)<n,$ $H_i^{k(i)+1/n}$ is a true null hypothesis, because the number of false null hypotheses in row $i,$  $k(i),$ is smaller than $k(i)+1.$ According to item 1 of Lemma \ref{lemma-meta}, if the assumptions of item 1(a) of Theorem \ref{meta-seq} hold, then for each $i\in\{1, \ldots, m\}$ such that $k(i)<n,$ the pair $(P_i^{(k(i)+1)/n}, |\mathcal{S}|_{\bm v})$ satisfies the dependency control condition with the identity shape function $\beta(x)=x.$  Therefore, according to item 1 of Lemma \ref{lemma-seq}, the sequential procedure in Section \ref{meta-analysis} with identity shape function guarantees (\ref{result}). Thus we have proved item 1(a) of Theorem \ref{meta-seq}. Item 1(b) follows similarly, based on item 2 of Lemma \ref{lemma-meta} and item 2 of Lemma \ref{lemma-seq}. 

Let us now prove item 2 of Theorem \ref{meta-seq}. Under the assumptions of this item, the $p$-values within each row are independent,  and the partial conjunction $p$-values are valid. The $p$-values across the rows are also independent, so for each $i\in\{1, \ldots, m\},$ $P_i^{(k(i)+1)/n}$ is independent of $\bm p_{(-i)\cdot},$ because  $P_i^{(k(i)+1)/n}$ is a combination of $p$-values in row $i.$ Let us show that in this case for each $i\in\{1, \ldots, n\}$ such that $k(i)<n,$ the pair $(P_i^{(k(i)+1)/n}, |\mathcal{S}^{(-i)\cdot}|_{\bm v})$ satisfies the dependency control condition with  the shape function $\beta(x)=x,$ yielding that item 2 of Lemma \ref{lemma-seq} with $\beta(x)=x$ holds. Let $i\in\{1, \ldots, n\}$ such that $k(i)<n.$ Then $P_i^{(k(i)+1)/n}$ is a $p$-value of a true null hypothesis, and it is stochastically lower bounded by a uniform variable on $[0,1].$ Since $P_i^{(k(i)+1)/n}$ is independent of $\bm p_{(-i)\cdot},$ the conditional distribution of $P_i^{(k(i)+1)/n}$ given $\bm p_{(-i)\cdot}$ is the same as its marginal distribution, and therefore is stochastically lower bounded by a uniform distribution. In addition, for a given $\bm p_{(-i)\cdot},$ $\mathcal{S}^{(-i)\cdot}$ is fixed. Hence we obtain for any $c>0$:
\begin{align*}\EE{\frac{\ind(P_i^{(k(i)+1)/n}\leq c|\mathcal{S}^{(-i)\cdot}|_{\bm v})}{|\mathcal{S}^{(-i)\cdot}|_{\bm v}}\,\, | \,\,\bm p_{(-i)\cdot}}\leq c,\end{align*}
therefore\begin{align*}
   & \EE{\frac{\ind(P_i^{(k(i)+1)/n})\leq c|\mathcal{S}^{(-i)\cdot}|_{\bm v}}{|\mathcal{S}^{(-i)\cdot}|_{\bm v}}}=\\&
   \EE{\EE{\frac{\ind(P_i^{(k(i)+1)/n})\leq c|\mathcal{S}^{(-i)\cdot}|_{\bm v}}{|\mathcal{S}^{(-i)\cdot}|_{\bm v}}\,\,|\,\,\bm p_{(-i)\cdot}}}\leq c.
\end{align*}
Thus we have shown that item 2 of Lemma \ref{lemma-seq} with $\beta(x)=x$ holds, therefore the result of item 2 of Theorem \ref{meta-seq} follows. Finally, item 3 of Theorem \ref{meta-seq} follows from item (iii) of Lemma 3.2 of \cite{blanchard2008two}, showing that the condition of item 1 of Lemma \ref{lemma-seq} holds with respect to the shape function $\beta$ of form (\ref{beta}).   
\begin{proof}[Proof of Lemma \ref{lemma-seq}]
Let us first prove item 1 of Lemma \ref{lemma-seq}. Define $I=\{i\in\{1, \ldots, m\}: k(i)<n\}.$ Note that the event $\{i\in \mathcal{S}, \hat{k}(i)>k(i)\}$ is equivalent to the event $\{i\in I\cap \mathcal{S}, \hat{k}(i)>k(i)\},$ because $\hat{k}(i)\leq n$ for each $i\in \mathcal{S}.$ Therefore,
\begin{align}\EE{\frac{\sum_{i\in \mathcal{S}}v_i\ind(\hat{k}(i)>k(i))}{|\mathcal{S}|_{\bm v}}}&=\EE{\frac{\sum_{i\in I}v_i\ind(i\in \mathcal{S}, \hat{k}(i)>k(i))}{|\mathcal{S}|_{\bm v}}}\notag\\&=\sum_{i\in I}v_i\EE{\frac{\ind(i\in \mathcal{S}, \hat{k}(i)>k(i))}{|\mathcal{S}|_{\bm v}}}\label{seq-1}\end{align}
For the sequential procedure with identity shape function, given in Section \ref{proc-rep}, the event $\{i\in \mathcal{S}, \hat{k}(i)>k(i)\}$ implies the event $\{i\in \mathcal{S}, P_i^{(k(i)+1)/n}\leq w_i|\mathcal{S}|_{\bm v}q/m\}.$ Indeed, by the definition of $\hat{k}(i),$ for each selected feature $i$ 
and for each $u\leq \hat{k}(i),$  $P_i^{u/n}\leq  w_i|\mathcal{S}|_{\bm v}q/m.$ If  $\hat{k}(i)>k(i),$ then $P^{(k(i)+1)/n}\leq w_i|\mathcal{S}|_{\bm v}q/m.$ Therefore, continuing from (\ref{seq-1}), we obtain
\begin{align}
&\sum_{i\in I} v_i\EE{\frac{\ind(i\in \mathcal{S}, \hat{k}(i)>k(i))}{|\mathcal{S}|_{\bm v}}}\leq\label{fornext}\\& \sum_{i\in I} v_i\EE{\frac{\ind(i\in \mathcal{S}, P_i^{(k(i)+1)/n}\leq w_i\beta(|\mathcal{S}|_{\bm v})q/m}{|\mathcal{S}|_{\bm v}}}\leq\notag\\&
\sum_{i\in I} v_i\EE{\frac{\ind(P_i^{(k(i)+1)/n}\leq w_i\beta(|\mathcal{S}|_{\bm v})q/m}{|\mathcal{S}|_{\bm v}}}\leq\label{seq-2}\\& \frac{q}{m}\sum_{i=1}^m v_i w_i=q\notag
\end{align} 
The inequality in (\ref{seq-2}) follows from the fact that for each $i\in\{1, \ldots, m\}$ such that $k(i)<n,$ $H_i^{(k(i)+1)/n}$ is a true null hypothesis, and the pair $$(P_i^{(k(i)+1)/n}, |\mathcal{S}|_{\bm v})$$ satisfies the dependency control condition with shape function $\beta.$ This completes the proof of item 1 of Lemma \ref{lemma-seq}. 

Let us prove item 2 of Lemma \ref{lemma-seq}. In this case $\mathcal{S}$ is stable, therefore $\mathcal{S}^{(-i)\cdot}$ is defined, and 
\begin{align*}
&\sum_{i\in I}v_i\EE{\frac{\ind(i\in \mathcal{S}, P_i^{(k(i)+1)/n}\leq w_i\beta(|\mathcal{S}|_{\bm v})q/m}{|\mathcal{S}|_{\bm v}}}=\\&\sum_{i\in I}v_i\EE{\frac{\ind(i\in \mathcal{S}, P_i^{(k(i)+1)/n}\leq w_i\beta(|\mathcal{S}^{(-i)\cdot}|_{\bm v})q/m}{|\mathcal{S}^{(-i)\cdot}|_{\bm v}}}.
\end{align*}
Therefore, using (\ref{seq-1}) and (\ref{fornext}), we obtain
\begin{align}&\EE{\frac{\sum_{i\in \mathcal{S}}v_i\ind(\hat{k}(i)>k(i))}{|\mathcal{S}|_{\bm v}}}\leq\notag\\& \sum_{i\in I}v_i\EE{\frac{\ind(i\in \mathcal{S}, P_i^{(k(i)+1)/n}\leq w_i\beta(|\mathcal{S}|_{\bm v})q/m}{|\mathcal{S}|_{\bm v}}}=\notag\\&
\sum_{i\in I}v_i\EE{\frac{\ind(i\in \mathcal{S}, P_i^{(k(i)+1)/n}\leq w_i\beta(|\mathcal{S}^{(-i)\cdot}|_{\bm v})q/m}{|\mathcal{S}^{(-i)\cdot}|_{\bm v}}}\leq\notag\\&\sum_{i\in I}v_i\EE{\frac{\ind(P_i^{(k(i)+1)/n}\leq w_i\beta(|\mathcal{S}^{(-i)\cdot}|_{\bm v})q/m}{|\mathcal{S}^{(-i)\cdot}|_{\bm v}}}\leq\label{seq-3}\\& \frac{q}{m}\sum_{i=1}^mv_iw_i=q,\notag
\end{align}
where the inequality in (\ref{seq-3}) follows from the fact that for each $i\in\{1, \ldots, m\}$ such that $k(i)<n,$ the pair $(P_i^{(k(i)+1)/n}, |\mathcal{S}^{(-i)}|_{\bm v})$ satisfies the dependency control condition with shape function $\beta.$
\end{proof}
\subsection{Proof of Lemma \ref{ov-PRDS1}}
The proof is based on the following auxiliary lemma.
\begin{lem}\label{help1-u}
Consider the setting of Section \ref{setting}.	Let $g\in \mathcal{G}_0,$ and assume that 
$P^{u_g/n_g}_g$ is of form (\ref{PCM}):  $$P_g^{u_g/n_g}=\max\{P^{1/(n_g-u_g+1)}[A]: A\subseteq A_g, |A|=n_g-u_g+1\},$$
where for each null group $A\subseteq A_g$ such that $|A|=n_g-u_g+1,$ $P^{1/(n_g-u_g+1)}[A]$ is a valid global null $p$-value for group $A.$
Let $U$ be a non-negative random variable, $\beta: \mathbb{R}^+\rightarrow  \mathbb{R}^+$ be a non-decreasing shape function satisfying $\beta(0)=0.$  Assume that for each null group $A\subseteq A_g$ such that $|A|=n_g-u_g+1,$ the pair $(P^{1/(n_g-u_g+1)}[A], \,U)$ satisfies the dependency control condition  with respect to the shape function $\beta.$   Then
	the pair    $(P^{u_g/n_g}_g, \, U)$ satisfies the dependency control condition with respect to the shape function $\beta.$ 
\end{lem}
The proof of this lemma is given in the end of this section. Note that the condition $\beta(0)=0$ holds for the shape functions considered in Lemmas \ref{ov-PRDS1} and \ref{lemma-meta}, i.e. for the identity shape function and for the shape function of form (\ref{beta}). The proof of Lemma \ref{ov-PRDS1} is based on the techniques of the proof of the group-level super-uniformity lemma of \cite{RetJ17}.
\paragraph{Proof of item 1.}
	Let $c$ be a positive constant, and $g\in \mathcal{G}_0.$ According to the assumptions of item 1, 
	$$P^{u_g/n_g}=\max\{P^{1/(n_g-u_g+1)}[A]: A\subseteq A_g, |A_g|=n_g-u_g+1\},$$
	where $P^{1/(n_g-u_g+1)}[A]$ is the 
	 minimum adjusted $p$-value based on the multiple testing procedure $\mathcal{M}_g$ applied on the $p$-values of hypotheses with indices in $A,$ and $\mathcal{M}_g$ is non-increasing and is self-consistent with respect to thresholds of form $\Delta(r)=r\alpha/m.$ Since it is assumed that the $p$-values satisfy the PRDS property on the subset of true null hypotheses, we obtain from Propositions 2.7 and 3.6 in \cite{blanchard2008two} that $\mathcal{M}_g$ controls the FDR when applied on any subset of $p$-values. Therefore, according to Corollary \ref{connection-cor}, for any null group   $A\subseteq A_g$ such that $|A_g|=n_g-u_g+1,$ $P^{1/(n_g-u_g+1)}[A]$  is a valid $p$-value for testing the global null for group $A,$ i.e. its distribution is uniform or is stochastically larger than uniform.

	Let $f:[0,1]^M\rightarrow [0, \infty)$ be a non-increasing function. According to Lemma \ref{help1-u}, it is enough to prove that for any null group   $A\subseteq A_g$ such that $|A_g|=n_g-u_g+1,$ the pair 
	 $$(P^{1/(n_g-u_g+1)}[A], f(\bm p)) $$ satisfies the dependency control condition with the shape function $\beta(x)=x.$ Let $A\subseteq A_g$ be a null group such that $|A_g|=n_g-u_g+1.$ Note that such group exists, because $g\in \mathcal{G}_0,$ yielding that the number of true null hypotheses in group $g$ is not smaller than $n_g-u_g+1.$ 
	 Let $c$ be a positive constant. We need to prove that
	 $$\EE{\frac{\ind(P^{1/(n_g-u_g+1)}[A]\leq cf(\bm p))}{f(\bm p)}}\leq c.$$ 
	Note that the event 	 $\{P^{1/(n_g-u_g+1)}[A]\leq cf(\bm p) \} $
is equivalent to the event in which $\mathcal{M}_g$ makes at least one rejection when it is applied on the $p$-values belonging to group $A$ at level $cf(\bm p).$
Let $\mathcal{R}_g^A$ be the set of indices of hypotheses rejected by $\mathcal{M}_g$ when it is applied on the $p$-values of group $A$ at level $cf(\bm p),$
and let $|\mathcal{R}_g^A|$ be their number.
Then 
\begin{align}\ind(\{P^{1/(n_g-u_g+1)}[A]\leq cf(P)\})=\ind(|\mathcal{R}_g^A|>0)=\frac{|\mathcal{R}_g^A|}{|\mathcal{R}_g^A|}.\label{self-cons-1} \end{align}
The last equality follows from the fact that we define $0/0=0,$ so the fraction on the right hand side is equal to 1 if $|\mathcal{R}_g^A|>0,$ and  is equal to 0 otherwise. Since $\mathcal{M}_g$ is self-consistent with respect to thresholds of form (\ref{non-adapt}) with unity weights and identity shape function, we obtain that $$\mathcal{R}_g^A\subseteq \{j\in A: p_j\leq |\mathcal{R}_g^A|cf(\bm p)/|A|\},$$
yielding that 
\begin{align}|\mathcal{R}_g^A|\leq \sum_{j\in A}\ind(p_j\leq |\mathcal{R}_g^A|cf(\bm p)/|A|).\label{self-cons-2}\end{align}
Note that if $|\mathcal{R}_g^{A}|=0,$ then $\sum_{j\in A}\ind(p_j\leq |\mathcal{R}_g^A|cf(\bm p)/|A|)=0$ almost surely, because for each $j\in A,$ $p_j$ is uniform or is stochastically larger than uniform, so $\PP{p_j\leq 0}=0.$ Therefore, according to Remark \ref{technical}, if $|\mathcal{R}_g^{A}|=0,$ the fraction
$$\frac{\sum_{j\in A}\ind(p_j\leq |\mathcal{R}_g^A|cf(\bm p)/|A|)}{|\mathcal{R}_g^A|}$$ is defined as 0.
It follows that we can divide both sides of (\ref{self-cons-2}) by $|R_g^A|,$ and the inequality will be preserved.
Therefore, 
combining (\ref{self-cons-1}) and (\ref{self-cons-2}), we obtain
\begin{align}
\ind(P^{1/(n_g-u_g+1)}[A]\leq cf(\bm p))\leq \frac{\sum_{j\in A}\ind(p_j\leq |\mathcal{R}_g^A|cf(\bm p)/|A|)}{|\mathcal{R}_g^A|}\label{last-self-cons}
\end{align}
Using the fact that $P^{1/(n_g-u_g+1)}[A]$ is uniform or is stochastically larger than uniform, and using the same argument as above, it is legitimate to divide  both expressions in (\ref{last-self-cons}) by $f(\bm p),$ preserving the inequality. Taking   expectations on both sides, we obtain:
\begin{align}
\EE{\frac{\ind(P^{1/(n_g-u_g+1)}[A]\leq cf(\bm p))}{f(\bm p)}}\leq \EE{ \frac{\sum_{j\in A}\ind(p_j\leq |\mathcal{R}_g^A|cf(\bm p)/|A|)}{|\mathcal{R}_g^A|f(\bm p)}}\label{key-self-cons}
\end{align}
Since $\mathcal{M}_g$ is non-increasing, $|\mathcal{R}_g^A|$ is non-increasing in each $p$-value in $A.$ In addition, $f(\bm p)$ is non-increasing in each $p$-value, therefore the function $g(\bm p)\equiv c|\mathcal{R}_g^A|f(\bm p)/|A|$ is non-increasing in each $p$-value. Since the $p$-values satisfy the PRDS property on the subset of true null hypotheses, and $A$ is a null group, we obtain from the super-uniformity lemma of \cite{RetJ17}, item 1, for each $j\in A,$
$$\EE{\frac{\ind(p_j\leq g(\bm p))}{g(\bm p)}}\leq 1.$$
Combining this inequality with (\ref{key-self-cons}), we obtain
\begin{align}
\EE{\frac{\ind(P^{1/(n_g-u_g+1)}[A]\leq cf(\bm p))}{f(\bm p)}}&\leq \EE{ \frac{\sum_{j\in A}\ind(p_j\leq |\mathcal{R}_g^A|cf(\bm p)/|A|)}{|\mathcal{R}_g^A|f(\bm p)}}\notag\\&=\sum_{j\in A} \EE{\frac{\ind(p_j\leq |\mathcal{R}_g^A|cf(\bm p)/|A|)}{|\mathcal{R}_g^A|f(\bm p)}}\label{int-1}\\&\leq\sum_{j\in A}\frac{c}{|A|}\EE{ \frac{\ind(p_j\leq g(\bm p))}{g(\bm p)}}\notag\\&\leq c\notag
\end{align}
This completes the proof. For the last step of the proof, we could rely on the results of \cite{blanchard2008two} rather on the superuniformity lemma of \cite{RetJ17}. Specifically, using the arguments in the proof of Proposition 3.6 in \cite{blanchard2008two}, one can show that for each $j\in A,$ the pair $(p_j,|\mathcal{R}_g^A|f(\bm p) )$ satisfies the dependency control condition with respect to the identity shape function $\beta(x)=x,$ which yields that for each $j\in A,$  $$\EE{\frac{\ind(p_j\leq |\mathcal{R}_g^A|cf(\bm p)/|A|)}{|\mathcal{R}_g^A|f(\bm p)}}\leq \frac{c}{|A|}.$$
Combining this inequality with (\ref{int-1}), the result follows.
\paragraph{Proof of item 2.}
Let $g\in \mathcal{G}_0$ and let $h: [0,1]^
{M-n_g}\rightarrow [0,\infty)$ be a non-increasing function. We need to prove that the pair $(P^{u_g/n_g}, h(\bm p^{-g}))$ satisfies the dependency control condition with respect to the shape function $\beta(x)=x.$ Based on Lemma \ref{help1-u}, it is enough to prove that for any null group $A\subseteq A_g$ such that $|A|=n_g-u_g+1,$ the pair $(P^{1/(n_g-u_g+1)}[A], \, h(\bm p^{-g}))$ satisfies the dependency control condition. 	Let $A\subseteq A_g$ be a null group such that $|A|=n_g-u_g+1.$ Such group exists since $g\in \mathcal{G}_0,$ as explained in the proof of item 1 of Lemma \ref{ov-PRDS1}. Let $c$ be a positive constant. According to the assumptions, $P^{1/(n_g-u_g+1)}[A]$ is the minimum adjusted $p$-value of a multiple testing procedure $\mathcal{M}_g$ applied on the $p$-values of $A,$
and $\mathcal{M}_g$ is non-increasing and self-consistent with respect to thresholds of form $\Delta(i, r)=r\alpha/(\hat{\pi}_0(\bm p^{A})|A|),$ where $\bm p^{A}$ is the vector of $p$-values for the hypotheses in group $A,$ and $\hat{\pi}_0(\bm p^{A})$ is an estimator for the proportion of true null hypotheses in group $A,$ $\pi_0^A.$ It is assumed that the function $\hat{\pi}_0$ satisfies Condition \ref{cond-adapt}.  
Since the $p$-values are assumed to be independent, according to Theorem 11 of \cite{blan2009} we obtain that $\mathcal{M}_g$ controls the FDR when applied on the $p$-values of $A,$ therefore $P^{1/(n_g-u_g+1)}[A]$ is a valid global null $p$-value for group $A.$
Let $\mathcal{R}_g^{A}$ be the set of rejections made by the multiple testing procedure $\mathcal{M}_g$ when it is applied on the $p$-values of group $A$ at level $ch(\bm p^{-g}),$ and let $|\mathcal{R}_g^{A}|$ be their number. Using similar arguments to those leading to (\ref{key-self-cons}), we obtain 
\begin{align}
&\EE{\frac{\ind(P^{1/(n_g-u_g+1)}[A]\leq ch(\bm p^{-g}))}{h(\bm p^{-g})}}\leq \notag\\&\sum_{i\in A}\EE{\frac{\ind(p_i\leq |\mathcal{R}_g^{A}|ch(\bm p^{-g})\hat{\pi}_{0}^{-1}(\bm p^{A})/|A|)}{h(\bm p^{-g})|\mathcal{R}_g^{A}|}},\label{zero-adapt}
\end{align}
 Let $i\in A$ be arbitrary fixed. Let $\bm p^A_{(0,i)}$ be the vector of $p$-values with indices in $A,$ where $p_i$ is replaced by 0. Since the function $\hat{\pi}_{0}^{-1}$ is non-increasing in each $p$-value,  $\hat{\pi}_{0}^{-1}(\bm p^A)\leq \hat{\pi}^{-1}(\bm p^A_{(0, i)}). $ Therefore,
\begin{align}
&\EE{\frac{\ind(p_i\leq |\mathcal{R}_g^A|ch(\bm p^{-g})\hat{\pi}_{0}^{-1}(\bm p_A)/|A|)}{h(\bm p^{-g})|\mathcal{R}_g^A|}}\leq\notag\\& \EE{\frac{\ind(p_i\leq |\mathcal{R}_g^A|ch(\bm p^{-g})\hat{\pi}_{0}^{-1}(\bm p^A_{(0, i)})/|A|)}{h(\bm p^{-g})|\mathcal{R}_g^A|}}.\label{first-adapt}
\end{align}
Recall that $\bm p_g^{-i}$ is the vector of the $p$-values in group $g,$ excluding $p_i.$ For fixed $\bm p_g^{-i},$ the $p$-values for the hypotheses in $A\setminus\{i\}$ are fixed, so $|\mathcal{R}_g^A|$ is a non-increasing function of $p_i,$ and $\hat{\pi}_{0}^{-1}(\bm p^A_{(0,i)})$ is fixed. In addition, $h(\bm p^{-g})$ is non-increasing in each entry of the vector $\bm p^{-g}.$ This yields that for fixed $\bm p_g^{-i},$  the function $c|\mathcal{R}_g^A|h(\bm p^{-g})\hat{\pi}_{0}^{-1}(\bm p_{(0, i)}^{A})/|A|$ is non-increasing in each coordinate of $\bm p.$ Since there is independence within each group, $p_i$ is independent of $\bm p_g^{-i},$ therefore the conditional distribution of $p_i$ given $\bm p_g^{-i}$ is uniform or is stochastically larger than uniform. Due to conditional positive dependence across groups, as defined in item D4 of Section \ref{setting},  the vector $\bm p$ is PRDS with respect to $p_i,$ conditionally on $\bm p_g^{-i}.$ Therefore, using item (b) of super-uniformity lemma of \cite{RetJ17},  we obtain
\begin{align*}
\EE{\frac{\ind(p_i\leq |\mathcal{R}_g^A|ch(\bm p^{-g})\hat{\pi}_{0}^{-1}(\bm p^A_{(0,i)})/|A|)}{|\mathcal{R}_g^A|ch(\bm p^{-g})\hat{\pi}_{0}^{-1}(\bm p^A_{(0,i)})/|A|}\,\,|\,\,\bm p_g^{-i}}\leq 1,
\end{align*}
which yields
\begin{align}&\EE{\frac{\ind(p_i\leq |\mathcal{R}_g^A|ch(\bm p^{-g})\hat{\pi}_{0}^{-1}(\bm p^A_{(0,i)})/|A|)}{h(\bm p^{-g})|\mathcal{R}_g^A|}\,\,|\,\,\bm p^{-i}_g}=\notag\\&\frac{c\hat{\pi}_{0}^{-1}(\bm p^A_{(0,i)})}{|A|}\EE{\frac{\ind(p_i\leq |\mathcal{R}_g^A|ch(\bm p^{-g})\hat{\pi}_{0}^{-1}(\bm p^A_{(0,i)})/|A|)}{|\mathcal{R}_g^A|ch(\bm p^{-g})\hat{\pi}_{0}^{-1}(\bm p^A_{(0,i)})/|A|}\,\,|\,\,\bm p_g^{-i}}\leq\\& \frac{c\hat{\pi}_{0}^{-1}(\bm p^A_{(0,i)})}{|A|}\label{intermediate}
\end{align}
Using  (\ref{intermediate}) we obtain 
\begin{align}&\EE{\frac{\ind(p_i\leq |\mathcal{R}_g^A|ch(\bm p^{-g})\hat{\pi}_{0}^{-1}(\bm p^A_{(0,i)})/|A|)}{h(\bm p^{-g})|\mathcal{R}_g^A|}}=\\&\EE{\EE{\frac{\ind(p_i\leq |\mathcal{R}_g^A|ch(\bm p^{-g})\hat{\pi}_{0}^{-1}(\bm p^A_{(0,i)})/|A|)}{h(\bm p^{-g})|\mathcal{R}_g^A|}}\,\,|\,\,\bm p_g^{-i}}\leq\notag\\&\frac{c}{|A|}\EE{\hat{\pi}_0^{-1}(\bm p^A_{(0,i)})}\label{last-adapt}
\end{align} 
Combining (\ref{last-adapt}) with (\ref{zero-adapt}) and (\ref{first-adapt}), we obtain
$$\EE{\frac{\ind(P^{1/(n_g-u_g+1)}\leq ch(\bm p^{-g}))}{h(\bm p^{-g})}}\leq\frac{c}{|A|}\sum_{i\in A}\EE{\hat{\pi}_0^{-1}(\bm p^A_{(0, i)})}.$$
Using the assumption that the estimator $\hat{\pi}_0$ satisfies Condition \ref{cond-adapt}, we obtain that for any $i\in A,$ $\EE{1/\hat{\pi}_0(\bm p^A_{(0,i)})}\leq 1/\pi_0^A.$ Since $A$ is a group consisting only of true null hypotheses, $\pi_0^A=1,$ therefore for any $i\in A,$ $\EE{\hat{\pi}_0^{-1}(\bm p^A_{(0,i)})}\leq 1.$ 
Therefore, we obtain
$$\EE{\frac{\ind(P^{1/(n_g-u_g+1)}[A]\leq ch(\bm p^{-g}))}{h(\bm p^{-g})}}\leq\frac{c}{|A|}\sum_{i\in A}\EE{\hat{\pi}_0^{-1}(\bm p^A_{(0,i)})}\leq c.$$
This completes the proof.
\paragraph{Proof of item 3.} Let $g\in \mathcal{G}_0$ and $h: [0,1]^{M-n_g}\rightarrow[0,\infty).$ Let $A\subseteq A_g$ be a null group such that $|A|=n_g-u_g+1$ (which exists, because $g\in \mathcal{G}_0$).
Based on Lemma \ref{help1-u}, it is enough prove that that the pair 
$(P^{1/(n_g-u_g+1)}[A], \, h(\bm p^{-g}))$ satisfies the dependency control condition
with respect to shape function $\beta(x)=x.$
	Let $c$ be a positive constant. Since 
	$P^{1/(n_g-u_g+1)}[A]$ is based on Bonferroni, $P^{1/(n_g-u_g+1)}[A]=|A|\min\{p_i: i\in A\}.$ Therefore, \begin{align}
\EE{\frac{\ind(P^{1/(n_g-u_g+1)}\leq ch(\bm p^{-g}))}{h(\bm p^{-g})}}&\leq 
\EE{\frac{\ind(\min\{p_i: i\in A\}\leq ch(\bm p^{-g})/|A|)}{h(\bm p^{-g})}}
\notag\\&\leq \sum_{i\in A}\EE{\frac{\ind(p_i\leq ch(\bm p^{-g})/|A|)}{h(\bm p^{-g})}}\label{prev}\\&=\sum_{i\in A}\frac{c}{|A|}\EE{\frac{\ind(p_i\leq ch(\bm p^{-g})/|A|)}{ch(\bm p^{-g})/|A|}}\label{zero-adapt1}
\end{align}
Let $i\in A$ be arbitrary fixed. Since $A$ consists only of true null hypotheses, $p_i$ corresponds to a true null hypothesis for each $i\in A.$ Therefore, the vector $(p_i, \bm p^{-g})$ is PRDS with respect to $p_i,$ according to the assumption of positive dependence across groups, as defined in item D5 in Section \ref{setting}. Using the fact that $ch(\bm p^{-g})/|A|$ is a coordinate-wise non-increasing function of $\bm p^{-g},$ we obtain from item (b) of the super-uniformity lemma of \cite{RetJ17}  that 
$$\EE{\frac{\ind(p_i\leq ch(\bm p^{-g})/|A|)}{ch(\bm p^{-g})/|A|}}\leq 1.$$
Combining this inequality with (\ref{zero-adapt1}), we obtain
\begin{align}
\EE{\frac{\ind(P^{1/(n_g-u_g+1)}[A]\leq ch(\bm p^{-g}))}{h(\bm p^{-g})}}&\leq 
 \sum_{i\in A}\frac{c}{|A|}=c,\label{final}
\end{align}
which completes the proof. As in item 1, we could use the results of \cite{blanchard2008two} rather than those of \cite{RetJ17} for the last step of the proof. Specifically, using similar arguments to those used in the proof of Proposition 3.6 of \cite{blanchard2008two}, we could show that for each $i\in A,$
$(p_i, h(\bm p^{-g}))$ satisfies the dependency control condition with respect to shape function $\beta(x)=x,$ which yields that
$$\EE{\frac{\ind(p_i\leq ch(\bm p^{-g})/|A|)}{h(\bm p^{-g})}}\leq \frac{c}{|A|}.$$ Using this result, we obtain (\ref{final}) from (\ref{prev}).
\begin{proof}[Proof of Lemma \ref{help1-u}]
Let $c$ be a positive constant and let $g\in \mathcal{G}_0.$ We need to show that 
$$\EE{\frac{\ind\{P_g^{u_g/n_g}\leq c\beta(U)\}}{U}}\leq c.$$
By assumption,
\begin{align}
P^{u_g/n_g}_g=\max\left\{P^{1/(n_g-u_g+1)}_g: A\subseteq A_g, |A|=n_g-u_g+1\right\}.\label{def}
\end{align}
Let $A_g^0$ be the indices of true null hypotheses belonging to group $g.$ Since $g\in \mathcal{G}_0,$ $H_0^{u_g/n_g}$ is a true null hypothesis, therefore
$|A_g^0|\geq n_g-u_g+1,$ and there exists a set $A\subseteq A_g^0$ such that $|A|=n_g-u_g+1.$ Thus we obtain
\begin{align}&\ind\left[\max\left\{P^{1/(n_g-u_g+1)}[A]: A\subseteq A_g, |A|=n_g-u_g+1\right\}\leq c\beta(U)\right]\leq\notag\\& \ind\left[\max\left\{P^{1/(n_g-u_g+1)}[A]: A\subseteq A_g^0, |A|=n_g-u_g+1\right\}\leq c\beta(U)\right]\label{subset}\end{align}
Combining (\ref{def}) and (\ref{subset}), we obtain
\begin{align}
&\EE{\frac{\ind\left\{P^{u_g/n_g}_g\leq c\beta(U)\right\}}{U}}\leq 
\notag\\& \EE{\frac{\ind\left(\max\{P^{1/(n_g-u_g+1)}[A]: A\subseteq A_g^0, |A|=n_g-u_g+1\}\leq c\beta(U)\right)}{U}}\label{first}
\end{align}
Let $\mathcal{B}=\{A\subseteq A_g^0: |A|=n_g-u_g+1\}$ be the set of all the subsets of $A_g^0$ of size $n_g-u_g+1.$ As shown above, this set is non-empty. It is easy to show that 
\begin{align}
    &\ind\left[\max\left\{P^{1/(n_g-u_g+1)}[A]: A\in \mathcal{B}\right\}\leq c\beta(U)\right]\leq \notag\\&\frac{1}{|\mathcal{B}|}\sum_{A\in \mathcal{B}}\ind\left(P^{1/(n_g-u_g+1)}[A]\leq c\beta(U)\right)\label{main-ineq1}
\end{align}
Indeed, since the expression on the right hand side is non-negative, it is enough to show the inequality above only for the case where the indicator on the left hand side is equal to one. In this case, for each $A\subseteq 
A_g^0$ such that $|A|=n_g-u_g+1,$ it holds that
$$P^{1/(n_g-u_g+1)}[A]\leq c\beta(U).$$ Therefore, in this case the expression on the right hand side is equal to 1, and thus we have proven the inequality in (\ref{main-ineq1}). According to our assumptions, for each $A\in \mathcal{B},$ the distribution of $P^{1/(n_g-u_g+1)}[A]$ is either uniform or is stochastically larger than uniform, so $\PP{P^{1/(n_g-u_g+1)}[A]\leq 0}=0.$ Therefore, the same is true for $\max\left\{P^{1/(n_g-u_g+1)}[A]: A\in \mathcal{B}\right\}.$ On the other hand, it is assumed that $\beta(0)=0.$ Therefore, according to Remark \ref{technical}, we can divide both sides of (\ref{main-ineq1}) while preserving the inequality. Dividing both expressions by $U,$ we obtain
\begin{align}
    &\frac{\ind\{(\max\{P^{1/(n_g-u_g+1)}[A]: A\in \mathcal{B}\}\leq c\beta(U))\}}{U}\leq \notag\\&\frac{1}{|\mathcal{B}|}\sum_{A\in \mathcal{B}}\frac{\ind\left(P^{1/(n_g-u_g+1)}[A]\leq c\beta(U)\right)}{U}\label{main-ineq}
\end{align}
Since it is assumed that for each $A\in\mathcal{B},$ the pair $(P_g^{1/(n_g-u_g+1)}[A], U)$ satisfies the dependency control condition with shape function $\beta,$ for each $A\in\mathcal{B}$ one has
$$\EE{\frac{\ind\left\{P_g^{1/(n_g-u_g+1)}[A]\leq c\beta(U)\right\}}{U}}\leq c.$$
Therefore, taking expectations of the expressions on both sides of inequality in (\ref{main-ineq}), we obtain
\begin{align}
    &\EE{\frac{\ind\left(\max\left\{P^{1/(n_g-u_g+1)}[A]: A\in \mathcal{B}\right\}\leq c\beta(U)\right)}{U}}\leq \notag\\&\frac{1}{|\mathcal{B}|}\sum_{A\in \mathcal{B}}\EE{\frac{\ind\left(P^{1/(n_g-u_g+1)}[A]\leq c\beta(U)\right)}{U}}\leq c\notag
\end{align}
Combining this result with (\ref{first}), we obtain
$$\EE{\frac{\ind\{P^{u_g/n_g}_g\leq c\beta(U)\}}{U}}\leq c ,$$
which completes the proof. 
\end{proof}
\subsection{Proof of Lemma \ref{lemma-meta}}
\paragraph{Proof of item 1.}
Let us first assume that $P_i^{u/n}$ is connected to a multiple testing procedure in the sense of (\ref{PCM}), which satisfies condition (c) of item 1 of Theorem \ref{min-adj-non-adapt}. The meta-analysis setting can be viewed as a group setting of Section \ref{setting}, where the $n$ hypotheses for each feature constitute a group. According to Lemma \ref{lemma-prob}, the $p$-values satisfy the overall positive dependence condition, as defined in item D3 in Section \ref{setting}. Therefore, according to item 1 of Lemma \ref{ov-PRDS1}, the pair $(P_i^{u/n}, f(\bm p))$ satisfies the dependency control condition with respect to the identity shape function $\beta(x)=x.$ 

Let us now assume that $P_i^{u/n}$ is based on Fisher's or on Stouffer's methods. In this case the proof is similar to the proof of Theorem 3 in \cite{BH08}. If  $P_i^{u/n}$ is based on Fisher's method, then 
\begin{align}P_i^{u/n}=\max\{P_i^{1/n-u+1}[A]: A\subseteq \{1, \ldots, n\}, |A|=n-u+1\},\label{def_PC}\end{align}
where 
\begin{align}P_i^{1/n-u+1}[A]=\PP{\chi^2_{2(n-u+1)}\geq -2\sum_{j\in A}\log p_{ij}}.\label{fish}\end{align}
If $P_i^{u/n}$ is based on Stouffer's method,
then $P_i^{u/n}$ is given by (\ref{def_PC}), where 
\begin{align}P_i^{1/n-u+1}[A]=\Phi\left\{\frac{\sum_{j\in A}\Phi^{-1}(p_{ij})}{\sqrt{n-u+1}}\right\}.\label{stouff}\end{align}
According to Lemma \ref{help1-u}, it is enough to show that for any subset of indices $A\subseteq \{1, \ldots, n\}$ such that $|A|=n-u+1,$ and $H_{ij}$ is a true null hypothesis for each $j\in A,$ the pair $(P_i^{1/n-u+1}[A], f(\bm p))$ satisfies the dependency control condition with  the identity shape function $\beta(x)=x,$ when $P_i^{1/n-u+1}[A]$ is given by (\ref{fish}) or (\ref{stouff}). Let $A\subseteq \{1, \ldots, n\}$ be such a subset (which exists, because $P_i^{u/n}$ is a true null hypothesis).
We shall rely on item 2 of Lemma 3.2 of \cite{blanchard2008two}, which we give below for completeness.
\begin{lem}[\cite{blanchard2008two}]\label{BR-08}
	Let $(U, V)$ be a couple of non-negative random variables such that $U$ is stochastically lower bounded by a uniform random variable on $[0,1],$ i.e. $\forall t\in[0,1],$ $\PP{U\leq t}\leq t.$ Then the pair $(U, V)$ satisfies the dependency control condition with the identity shape function $\beta(x)=x$ if for any $r\geq 0,$ the function $u\mapsto \PP{V<r\,|\,U\leq u}$ is non-decreasing.
\end{lem} 
The set $A$ consists only of true null hypotheses, and according to our assumptions, their $p$-values are independent and uniform on $[0,1].$ Therefore, $P_i^{1/n-u+1}[A]$ based on Fisher's or Stouffer's methods (given in (\ref{fish}) and (\ref{stouff}), respectively) is a a uniform random variable on $[0,1].$ According to Lemma \ref{BR-08} it is enough to show that for any $r\geq 0,$ the function $$u\mapsto \PP{f(\bm p)<r\,|\,P^{1/n-u+1}[A]\leq u}$$ is non-decreasing, where $P_i^{1/n-u+1}[A]$ is based on either Fisher's or Stouffer's methods. As shown by \cite{BY01} and \cite{blanchard2008two}, the above result follows if for any $r\geq 0,$ the function $u\mapsto \PP{f(\bm p)<r\,|\,P^{1/n-u+1}[A]= u}$ is non-decreasing, so it is enough to prove the latter. Let $r\geq 0$ be arbitrary fixed. 
Define $W:[0,1]^{n-u+1}\rightarrow [0,1]$ as follows:
$$W(x_1,\ldots, x_{n-u+1})=\PP{f(\bm p)<r\,|\,\forall k\in A:p_{ik}=x_k}.$$ Note that the set $D=\{\bm p: f(\bm p)<r\}$ is a non-decreasing set, since $f$ is a non-increasing function. Therefore, using the result of item 3 of Lemma \ref{lemma-prob} and the fact that the $p$-values within each row are independent,  we obtain that $W(x_1, \ldots, x_{n-u+1})$ is non-decreasing in each of its arguments. 
Note that for any $u\in [0,1],$ \begin{align}\PP{f(\bm p)<r\,|\,P^{1/n-u+1}[A]= u}=\EE{W(\bm p_i^{A})\,|\,P^{1/n-u+1}[A]=u},\label{main-w}\end{align} where $\bm p_i^{A}$
is the vector of $p$-values for the hypotheses in the set $\{H_{ij}, j\in A\}.$ We shall now use 
the following theorem of \cite{efron1965increasing}, addressing random variables with densities which are Polya frequency functions of order 2 ($PF_2$):
\begin{thm}[\cite{efron1965increasing}]\label{Efron}
	Let $X_1, \ldots, X_n$ be $n$ independent random variables with $PF_2$ densities $r_1(x), \ldots, r_n(x)$ respectively, and let $H(x_1, \ldots, x_n)$ be a real measurable function on Eucledean $n-$space which is non-decreasing in each of its arguments. Then $\EE{H(X_1, \ldots, X_n)|\sum_{i=1}^nX_i=y} $ is a non-decreasing function of $y.$
\end{thm}
Note that when $P_i^{1/(n-u+1)}[A]$ is based on either Fisher's or Stouffer's methods, $P_i^{1/(n-u+1)}[A]=G(\sum_{j\in A}g(p_{ij})),$ where $G(\cdot)$ and $g(\cdot)$ are certain strictly increasing functions. For Fisher's method, $G(x)=\PP{\chi^2_{2(n-u+1)}\geq -2x}$ and $g(x)=\log(x),$ while for Stouffer's method, $G(x)=\Phi(x/\sqrt{n-u+1})$ and $g(x)=\Phi^{-1}(x).$ 
Recall that $A$ consists only of true null hypotheses, therefore, according to our assumptions, the $p$-values in the set $\{p_{ij}, j\in A\}$ 
are independent and uniform on $[0, 1]$. 
Let us consider the random variables in the set $\{g(p_{ij}),\, j\in A\}.$ These random variables are independent. For each $j\in A,$ $g(p_{ij})$ is distributed as $\log(U)$ for Fisher's method, and $g(p_{ij})$ is distributed as $\Phi^{-1}(U)$  for Stouffer's method, where $U\sim U[0,1].$ Note that $-\log(U)$ has an exponential distribution, while $\Phi^{-1}(U)$ has a standard normal distribution. Therefore, for each $j\in A,$ $g(p_{ij})$ has a $PF_2$ density (see \cite{efron1965increasing}).
We shall consider the general case, where \begin{align}P_i^{1/(n-u+1)}[A]=G\left[\sum_{j\in A}g(p_{ij})\right],\label{main-g-fish}\end{align} where both $G(\cdot)$ and $g(\cdot)$ are strictly increasing functions, and the variables in the set $\{g(p_{ij}),\, j\in A\}$ are independent variables with $PF_2$ densities. As shown above, this case covers both Fisher's and Stouffer's combining functions. In this case the inverse functions $G^{-1}(\cdot)$ and $g^{-1}(\cdot)$ exist, and they are also strictly increasing. We define $H(x_1, \ldots, x_n)=W(g^{-1}(x_1), \ldots, g^{-1}(x_n)).$ Obviously,
$H(g(x_1), \ldots, g(x_n))=W(x_1, \ldots, x_n).$ Since $W(x_1, \ldots, x_n)$ is non-decreasing in each of its arguments, and $g^{-1}$ is strictly increasing, $H(x_1, \ldots, x_n)$ is also non-decreasing in each of its arguments.  Let $g(\bm p_i^{A})$ be the vector obtained by applying the function $g$ on each of its entries. Based on (\ref{main-w}) and (\ref{main-g-fish}), we obtain
\begin{align}&
\PP{f(\bm p)<r\,|\,P_i^{1/n-u+1}[A]= u}=\notag\\&\EE{W(\bm p_i^{A})\,|\,G\left\{\sum_{j\in A}g(p_{ij})\right\}=u}=\notag\\&\EE{H(g(\bm p_{i}^{A}))\,|\,\sum_{j\in A}g(p_{ij})= G^{-1}(u)}.\label{last-efron}
\end{align}
According to Theorem \ref{Efron}, the expression in (\ref{last-efron}) is non-decreasing in $G^{-1}(u),$ and since $G^{-1}$ is an increasing function, this expression is non-decreasing in $u.$ Thus we have proved that 
$$\PP{f(\bm p)<r\,|\,P_i^{1/n-u+1}[A]= u} $$
is non-decreasing in $u.$ As shown above,
this completes the proof.
\paragraph{Proof of item 2.}
Assume that $P_i^{u/n}$ is connected to a multiple testing procedure (in the sense of (\ref{PCM})), which satisfies condition (c) of item 2 of Theorem \ref{min-adj-non-adapt}. The meta-analysis setting can be viewed as a group setting of Section \ref{setting}, where the $n$ hypotheses for each feature constitute a group. According to Lemma \ref{lemma-prob}, the $p$-values are conditionally positively dependent across groups, and are independent within each group, in the sense of items D4 and D2  in Section \ref{setting}. Therefore, according to item 2 of Lemma \ref{ov-PRDS1}, the pair $(P_i^{u/n}, h(\bm{p}_{(-i)\cdot}))$ satisfies the dependency control condition with the identity shape function $\beta(x)=x.$ 
\bibliographystyle{plainnat}
\bibliography{Mybib}
\label{refs}
\end{document}